\documentclass[12pt]{article}

\usepackage[margin=1in]{geometry}
\usepackage{mathpazo}
\usepackage{stmaryrd}
\usepackage[english]{babel}
\usepackage[autostyle, english = american]{csquotes}
\usepackage{anyfontsize}
\usepackage{authblk}

\usepackage{tikz}

\usetikzlibrary{
  arrows, 
  automata, 
  shapes, 
  shapes.geometric, 
  positioning,
  calc, 
  chains, 
  decorations.pathmorphing,
  intersections, fit,cd}

\usepackage{hyperref}
\hypersetup{pdfpagemode=UseNone}

\usepackage{amssymb, amsthm, amsmath, dsfont, mathtools}
\mathtoolsset{showonlyrefs}
\usepackage{latexsym}
\usepackage{eepic}
\usepackage{sectsty}
\usepackage{framed}
\usepackage[shortlabels]{enumitem}
\usepackage{float}
\restylefloat{table}
\usepackage[pagecolor=white]{pagecolor}
\usepackage{xcolor}
\newcommand{\globalcolor}[1]{%
  \color{#1}\global\let\default@color\current@color
}
\newtheorem{theorem}{Theorem}
\newtheorem{lemma}[theorem]{Lemma}
\newtheorem{prop}[theorem]{Proposition}
\newtheorem{cor}[theorem]{Corollary}

\newtheorem{question}[theorem]{Question}
\theoremstyle{definition}
\newtheorem{definition}[theorem]{Definition}
\newtheorem{claim}[theorem]{Claim}

\newtheorem{fact}[theorem]{Fact}
\newtheorem{example}[theorem]{Example}

\newcommand{\rank}{\operatorname{rank}}
\newcommand{\krank}{\operatorname{k-rank}}

\newcommand{\abs}[1]{\lvert #1 \rvert}
\newcommand{\bigabs}[1]{\bigl\lvert #1 \bigr\rvert}

\newcommand{\ceil}[1]{\lceil #1 \rceil}
\newcommand{\bigceil}[1]{\bigl\lceil #1 \bigr\rceil}
\newcommand{\Bigceil}[1]{\Bigl\lceil #1 \Bigr\rceil}
\newcommand{\biggceil}[1]{\biggl\lceil #1 \biggr\rceil}

\newcommand{\floor}[1]{\lfloor #1 \rfloor}

\newcommand{\I}{\mathds{1}}

\newcommand{\setft}[1]{\mathrm{#1}}

\newcommand{\Lin}{\setft{L}}

\newcommand{\pro}[1]{\setft{Prod}\left(#1 \right)}

\newcommand{\complex}{\mathbb{C}}
\newcommand{\field}{\mathbb{F}}
\newcommand{\real}{\mathbb{R}}

\newenvironment{namedtheorem}[1]
	       {\begin{trivlist}\item {\bf #1.}\em}{\end{trivlist}}

\newcommand\W{\mathcal{W}}

\newcommand\V{\mathcal{V}}

\newcommand\C{\mathcal{C}}

\newcommand\R{\mathcal{R}}

\DeclareMathOperator{\spn}{span}

\usepackage{xifthen,mathrsfs}

\newcommand{\eql}[2]{\begin{align}\label{#1}#2\end{align}}
\newcommand{\eq}[2]{
\ifthenelse{\equal{#1}{}}{\begin{align}#2\end{align}}{\eql{#1}{#2}}}

\newcommand{\ha}[2][]{
\ifthenelse{\equal{#1}{}}{#2}{#1, #2}}
\newcommand{\Char}{\setft{Char}}

\begin{document}
\color{black}
\emergencystretch 3em
\title{\bf A generalization of Kruskal's theorem on tensor decomposition}

\author[$\dagger$]{
  Benjamin Lovitz\thanks{emails: benjamin.lovitz@gmail.com, f.v.petrov@spbu.ru}} 
 \author[$* \S \ddagger$]{Fedor Petrov}
\affil[$\dagger$]{Institute for Quantum Computing and Department of Applied Mathematics, 

University of Waterloo, Canada}
\affil[$\ddagger$]{St. Petersburg State University, St. Petersburg, Russia.}
\affil[$\S$]{St. Petersburg Department of Steklov Mathematical Institute of Russian Academy of Sciences,
St. Petersburg, Russia.}

\maketitle
\begin{abstract}
Kruskal's theorem states that a sum of product tensors constitutes a unique tensor rank decomposition if the so-called \textit{k-ranks} of the product tensors are large. We prove a ``splitting theorem'' for sets of product tensors, in which the k-rank condition of Kruskal's theorem is weakened to the standard notion of rank, and the conclusion of uniqueness is relaxed to the statement that the set of product tensors splits (i.e. is disconnected as a matroid). Our splitting theorem implies a generalization of Kruskal's theorem. While several extensions of Kruskal's theorem are already present in the literature, all of these use Kruskal's original permutation lemma, and hence still cannot certify uniqueness when the k-ranks are below a certain threshold. Our generalization uses a completely new proof technique, contains many of these extensions, and can certify uniqueness below this threshold. We obtain several other useful results on tensor decompositions as consequences of our splitting theorem. We prove sharp lower bounds on tensor rank and Waring rank, which extend Sylvester's matrix rank inequality to tensors. We also prove novel uniqueness results for non-rank tensor decompositions.
%A special case of our splitting theorem says that if $n$ product tensors form a circuit, then they have rank greater than one in at most $n-2$ subsystems. This corollary strengthens several recent results in this direction, and is sharp.
\end{abstract}

\newpage
\tableofcontents
\newpage
%-----------------------------------------------------------------------------%
\section{Introduction}\label{intro}

% so if true it could yield a novel alternate proof of this fundamental result. Note that several alternate proofs of Kruskal's theorem are already present in the literature \cite{Jiang2004KruskalsPL,STEGEMAN2007540,RHODES20101818,landsberg2012tensors}.

Let $[m]=\{1,\dots, m\}$ when $m$ is a positive integer, and let $[0]=\{\}$ be the empty set. For vector spaces $\V_1,\dots, \V_m$ over a field $\field$, a \textit{product tensor} in $\V=\V_1\otimes\dots \otimes \V_m$ is a non-zero tensor $z \in \V$ of the form ${z=z_1\otimes \dots \otimes z_m}$, with $z_j \in \V_j$ for all $j \in [m]$. We refer to the spaces $\V_j$ that make up the space $\V$ as \textit{subsystems}.
% (also known as \textit{factors} and \textit{loadings}).
The \textit{tensor rank} (or \textit{rank}) of a tensor $v \in \V$, denoted by $\rank(v)$, is the minimum number $n$ for which $v$ is the sum of $n$ product tensors. A decomposition of $v$ into a sum of $\rank(v)$ product tensors is called a \textit{tensor rank decomposition} of $v$.
%(also known as a \textit{canonical decomposition (CANDECOMP), parallel factor (PARAFAC) model, canonical polyadic (CP) decomposition,} and \textit{topographic components model}).
An expression of $v$ as a sum of product tensors (not necessarily of minimum number) is known simply as a \textit{decomposition} of $v$. A decomposition of $v$
\begin{align}\label{decomposition_of_v}
v=\sum_{a \in [n]} x_a
\end{align}
into a sum of product tensors $\{x_a : a \in [n]\}$ is said to be the \textit{unique tensor rank decomposition} of $v$ if for any decomposition
\begin{align}\label{other_decomp}
v=\sum_{a \in [r]} y_a
\end{align}
of $v$ into the sum of $r \leq n$ product tensors $\{y_a : a \in [r]\}$, it holds that $r=n$ and ${\{x_a : a \in [n]\}}={\{y_a : a \in [n]\}}$ as multisets. The decomposition~\eqref{decomposition_of_v} is said to be \textit{unique in the $j$-th subsystem} if for any other decomposition~\eqref{other_decomp}, it holds that $r=n$ and there exists a permutation $\sigma \in S_n$ such that $x_{a,j} \in \spn \{y_{\sigma(a),j}\}$ for all $a \in [n]$. Kruskal's theorem gives sufficient conditions for a given decomposition to constitute a unique tensor rank decomposition~\cite{kruskal1977three}. We refer to results of this kind as \textit{uniqueness criteria}.

Uniqueness criteria have found scientific applications in signal processing and spectroscopy, among others
%\cite{smilde2004multi,kroonenberg2008applied,doi:10.1137/07070111X,art1,comon2010handbook,art2,cichocki2015tensor,sidiropoulos2017tensor}
\cite{
%smilde2004multi,kroonenberg2008applied,doi:10.1137/07070111X,art1,comon2010handbook,
art2, landsberg2012tensors, cichocki2015tensor,sidiropoulos2017tensor}. In these circles, subsystems are also referred to as \textit{factors} and \textit{loadings}, and the tensor rank decomposition is also referred to as the \textit{canonical decomposition (CANDECOMP), parallel factor (PARAFAC) model, canonical polyadic (CP) decomposition,} and \textit{topographic components model}. Uniqueness of a tensor decomposition is also referred to as \textit{specific identifiability}, and uniqueness criteria as \textit{identifiability criteria}.

\subsection{Kruskal's theorem, and a generalization}
%Kruskal's theorem~\cite{kruskal1977three} was a seminal result that determined sufficient conditions for a decomposition~\eqref{decomposition_of_v} to constitute a unique tensor rank decomposition. We refer to results of this kind as \textit{uniqueness conditions}.
%\cite{kruskal1977three,1323268,10.1137/040608830,stegeman2009uniqueness,doi:10.1137/090779632,domanov2013uniqueness,domanov2013uniqueness2,domanov2014canonical,sorensen2015new,sorensen2015coupled,domanov2017canonical}.
For a finite set $S$, let $\abs{S}$ be the size of $S$. The \textit{Kruskal-rank} (or \textit{k-rank}) of a multiset of vectors $\{u_1,\dots, u_n\}$, denoted by ${\krank(u_1,\dots, u_n)}$, is the largest number $k$ for which $\dim\spn\{ u_a : a \in S\} =k$ for every subset $S \subseteq [n]$ of size $\abs{S}=k$. Similarly, we call $\dim\spn\{ u_a : a \in [n]\}$ the \textit{standard rank} (or \textit{rank}) of $\{u_1,\dots, u_n\}$.
%\begin{definition}
%Let $n$, $m$ and $k_1, \dots, k_m$ be positive integers, and let $\V=\V_1\otimes \dots\otimes \V_m$ be a multipartite vector space over a field $\mathbb{F}$. We say that a set of product tensors
%\begin{align}\label{product_tensors}
%\{{x_{a,1}}\otimes\dots\otimes x_{a,m}: a \in [n] \}\subseteq \V
%\end{align}
%is in $(k_1,\dots,k_m)$\textit{-general position} if for each index $j\in [m]$, the vectors $\{x_{a,j}: a \in [n]\}$ have \textit{Kruskal rank} (or \textit{k-rank}) at least $k_j$.
%\end{definition}
Kruskal's theorem states that if a collection of product tensors $\{x_{a,1}\otimes \dots \otimes x_{a,m} : a \in [n]\}$ has large enough k-ranks $k_j=\krank(x_{1,j},\dots, x_{n,j})$, then their sum constitutes a unique tensor rank  decomposition. This theorem was originally proven for $m=3$ subsystems over $\real$ \cite{kruskal1977three}, was later extended to more than three subsystems by Sidiropoulos and Bro \cite{sidiropoulos2000uniqueness}, and then extended to an arbitrary field by Rhodes \cite{RHODES20101818}.
\begin{theorem}[Kruskal's theorem]\label{kruskal}
Let $n \geq 2$ and $m \geq 3$ be integers, let ${\V=\V_1\otimes \cdots \otimes \V_m}$ be a vector space over a field $\field$, and let
\begin{align}
\{{x_{a,1}}\otimes\dots\otimes x_{a,m}: a \in [n] \}\subseteq \V\setminus\{0\}
\end{align}
be a multiset of product tensors. For each $a \in [n]$, let $x_a={x_{a,1}}\otimes\dots\otimes x_{a,m}.$ For each $j \in [m]$, let
\begin{align}
k_j = \krank(x_{1,j},\dots, x_{n,j}).
\end{align}
If ${2n \leq\sum_{j=1}^m (k_j-1)+1},$ then $\sum_{a \in [n]} x_{a}$ constitutes a unique tensor rank decomposition.
\end{theorem}

%In \cite{Jiang2004KruskalsPL,STEGEMAN2007540,RHODES20101818,landsberg2012tensors} alternate proofs of Kruskal's theorem are given.

%In \cite{Berge:2002aa} it is shown that Kruskal's inequality $2n \leq \sum_{j=1}^m (k_j-1)+1$ is also necessary for uniqueness when $n \leq 3$ and $m=3$, but not for $n\geq 4$. However, it is also shown that Kruskal's inequality is necessary for uniqueness even for $m=3$, $n=4$ under the condition that
%\begin{align}
%\dim\spn \{x_{a,j} : a \in [n]\}=k_j\quad\text{for all}\quad j \in [3].
%\end{align}
%In \cite{STEGEMAN2006210} it is shown that Kruskal's inequality is not necessary for uniqueness, even under this restrictive condition, for $n\geq 5$.

In \cite{DERKSEN2013708} it is shown that the inequality appearing in Kruskal's theorem cannot be weakened: there exist cases in which ${2n = \sum_{j=1}^m (k_j-1)+2}$ and the decomposition is not unique. While Kruskal's theorem gives sufficient conditions for uniqueness, necessary conditions are obtained in \cite{krijnen1993analysis,STRASSEN1983645,liu2001cramer}. In \cite{effective} it is shown that Kruskal's theorem is \textit{effective} over $\real$ or $\complex$ in the sense that it certifies uniqueness on a dense open subset of the smallest semialgebraic set containing the set of rank $n$ tensors. A robust form of Kruskal's theorem is proven in~\cite{bhaskara2014uniqueness}.

%Generic uniqueness has been studied, for example, in \cite{Bocci:2014aa,chiantini,lathauwergeneric}. Uniqueness of Waring rank decompositions (see Section~\ref{waring_rank}), and of other types of decompositions have been studied, for example, in \cite{article,Ballico:2012aa,sorensen2015coupled,Massarenti:2018aa,Chiantini2019,Angelini:2020aa}.

Our main result in this work is a ``splitting theorem," which is not itself a uniqueness criterion, but implies a criterion that generalizes Kruskal's theorem. In our splitting theorem, the k-rank condition in Kruskal's theorem is relaxed to a standard rank condition. In turn, the conclusion is also relaxed to a statement describing the linear dependence of the product tensors. Before stating our splitting theorem, we first introduce the generalization of Kruskal's theorem it implies.

%A drawback to Kruskal's theorem, and to all of the known uniqueness conditions that we are aware of, is that they cannot certify uniqueness when the $k$-ranks are small. The following conjectural statement would generalize Kruskal's theorem and has the capacity to efficiently certify uniqueness even when the k-ranks are small.

\begin{theorem}[Generalization of Kruskal's theorem]\label{k-gen}
Let $n \geq 2$ and $m \geq 3$ be integers, let $\V=\V_1 \otimes \cdots \otimes \V_m$ be a vector space over a field $\field$, and let
\begin{align}
\{{x_{a,1}}\otimes\dots\otimes x_{a,m}: a \in [n] \}\subseteq \V\setminus \{0\}
\end{align}
be a multiset of product tensors. For each $a\in [n]$, let $x_a = {x_{a,1}}\otimes\dots\otimes x_{a,m}$. For each subset $S \subseteq [n]$ and index $j\in [m]$, let
\begin{align}
d_j^S=\dim\spn \{ x_{a,j}: a \in S\}.
\end{align}
If $\;2 \abs{S} \leq\sum_{j=1}^m (d_j^S-1)+1$ for every subset $S\subseteq [n]$ with $2\leq \abs{S} \leq n$, then $\sum_{a \in [n]} x_a$ constitutes a unique tensor rank decomposition.
\end{theorem}

Note that the computational cost of checking the conditions of our Theorem~\ref{k-gen} is essentially the same as that of checking the conditions of Kruskal's theorem. In both cases, the quantities $d_j^S$ must be computed for all $j \in [m]$ and $S\subseteq [n]$ with $2\leq \abs{S} \leq n$. To verify Kruskal's conditions, one uses these quantities to compute the Kruskal ranks, and then checks the single inequality $2n \leq \sum_{j=1}^m (k_j-1)+1$. To verify the conditions of our generalization, one checks a separate inequality $\;2 \abs{S} \leq\sum_{j=1}^m (d_j^S-1)+1$ for every $S$.

To see that Theorem~\ref{k-gen} contains Kruskal's theorem, assume the conditions of Kruskal's theorem hold and note that for any subset $S \subseteq [n]$, the multiset of product tensors $\{x_a : a \in S\}$ satisfies $d_j^S \geq \min\{k_j, \abs{S}\}$. Using this fact, it is easy to verify that ${2\abs{S} \leq \sum_{j=1}^m (d_j^S-1)+1}$ for every subset $S \subseteq [n]$ with $2 \leq \abs{S} \leq n$.

In Section~\ref{uniqueness_applications} we compare Theorem~\ref{k-gen} to the uniqueness criteria of Domanov, De Lathauwer, and S\o{}rensen (DLS), which are the only known extensions of Kruskal's theorem that we are aware of \cite{domanov2013uniqueness,domanov2013uniqueness2,domanov2014canonical,sorensen2015new,sorensen2015coupled}. All of these extensions rely on Kruskal's original permutation lemma, and as a result, still require the k-ranks to be above a certain threshold. Our generalization uses a completely new proof technique, can certify uniqueness below this threshold, and contains many of these extensions. The cited results of DLS contain many similar but incomparable criteria, which can be difficult to keep track of. For clarity and future reference, in Theorem~\ref{uniqueness} we synthesize these criteria into a single statement. Using insight gained from this synthesization and our generalization of Kruskal's theorem, we propose a conjectural uniqueness criterion that would contain and unify every uniqueness criteria of DLS into a single, elegant statement.

For $m\geq 4$, Kruskal's theorem can be ``reshaped" by regarding multiple subsystems as a single subsystem. In Section~\ref{k-gen_proof} we present an analogous reshaping of Theorem~\ref{k-gen}, which has many more degrees of freedom to choose from than the reshaped Kruskal's theorem.

\subsection{A splitting theorem for product tensors}

We now state our splitting theorem, which we use in Section~\ref{k-gen_proof} to prove our generalization of Kruskal's theorem, and in Sections~\ref{interpolate},~\ref{waring_rank}, and~\ref{symmetric_non-rank} to obtain further results on tensor decompositions. We first require a definition.
\begin{definition}
Let $n\geq 2$ be an integer, and let $\V$ be a vector space over a field $\field$. We say that a multiset of non-zero vectors $\{v_1,\dots,v_n\}\subseteq \V\setminus \{0\}$ \textit{splits}, or is \textit{disconnected}, if there exists a subset $S \subseteq \{v_1,\dots,v_n\}$ with $1 \leq \abs{S}\leq n-1$ for which
\begin{align}\label{sdef}
\spn\{v_1,\dots,v_n\}=\spn(S) \oplus \spn(S^c),
\end{align}
where $S^c:=\{v_1,\dots, v_n\} \setminus S$. In this case, we say that $S$ \textit{separates} $\{v_1,\dots, v_n\}$.
If $\{v_1,\dots,v_n\}$ does not split, then we say it is \textit{connected}.
\end{definition}
Note that $\{v_1,\dots,v_n\}$ splits if and only if it is disconnected as a matroid~\cite{oxley2006matroid}. We now state our main result.

%Here we conjecture an analogous statement to Kruskal's theorem.
%%prove it in two special cases, and show that it would imply Kruskal's theorem as a corollary. We also use Kruskal's theorem to prove a family of statements which contain recent results in \cite{1751-8121-48-4-045303}, and show that our conjecture would imply an analogous family of statements for vectors 
%To introduce our conjecture and it's connection to Kruskal's theorem, we first state (an equivalent reformulation of) Kruskal's theorem.

\begin{theorem}[Splitting theorem]\label{conjecture}
Let $n \geq 2$ and $m \geq 2$ be integers, let $\V=\V_1\otimes \dots \otimes \V_m$ be a vector space over a field $\field$, let
\begin{align}
E=\{x_{a,1} \otimes \dots \otimes x_{a,m}: a \in[n] \} \subseteq \V\setminus\{0\}
\end{align}
be a multiset of product tensors, and for each $j \in [m]$, let
\begin{align}
d_j=\dim\spn \{ x_{a,j}: a \in [n]\}.
\end{align}
If $\dim\spn(E)\leq \sum_{j=1}^m (d_j-1)$, then $E$ splits.
\end{theorem}

In Section~\ref{conjecture_kruskal} we use Derksen's result \cite{DERKSEN2013708} to prove that the inequality appearing in Theorem~\ref{conjecture} cannot be weakened.

We now give a rough sketch of how our splitting theorem implies Theorem~\ref{k-gen}, which we formalize in Section~\ref{k-gen_proof}. First, a direct consequence of Theorem~\ref{conjecture} is that $E$ splits whenever $n \leq \sum_{j=1}^m (d_j-1)+1$ (see Corollary~\ref{original_conjecture}). To prove Theorem~\ref{k-gen}, let $\{x_a : a \in [n]\}$ be a multiset of product tensors satisfying the assumptions of Theorem~\ref{k-gen}, and let $\{y_a : a \in [r]\}$ be a multiset of $r\leq n$ product tensors for which $\sum_{a \in [n]} x_a = \sum_{a \in [r]} y_a$. Consider the multiset of $[n+r]$ product tensors
\begin{align}
E=\{x_a : a \in [n]\} \cup \{-y_a : a \in [r]\}.
\end{align}
Since $2n \leq \sum_{j=1}^m (d_j^{[n]}-1)+1$, $E$ splits. Since $\Sigma(E)=0$, it follows that $\Sigma(S)=\Sigma(S^c)=0$ for any separator $S$ of $E$. Now, continue applying the splitting theorem to $S$ and $S^c$, until every multiset has size 2, and contains one element each of ${\{x_a : a \in [n]\}}$ and ${\{-y_a : a \in [r]\}}$.

\subsection{Further applications of the splitting theorem to tensor decompositions}

In Sections~\ref{interpolate},~\ref{waring_rank}, and~\ref{symmetric_non-rank} we use the splitting theorem to prove further uniqueness results and sharp lower bounds on tensor rank. In Section~\ref{interpolate} we prove a general statement that interpolates between our generalization of Kruskal's theorem and a natural offshoot of our splitting theorem (mentioned above), obtaining uniqueness results for weaker notions of uniqueness. In Section~\ref{waring_rank} we prove sharp lower bounds on tensor rank and \textit{Waring rank}, a notion of rank for symmetric tensors. In Sections~\ref{interpolate} and~\ref{symmetric_non-rank} we obtain uniqueness results for \textit{non-rank} decompositions, a novel concept introduced in this work. We close this introduction by reviewing these results in more detail.

It is known that if a multiset of product tensors ${\{x_a : a\in[n]\}}$ satisfies
\begin{align}\label{intro_eq}
{n+r \leq \sum_{j=1}^m (k_j-1)+1}
\end{align}
for $r=0$, then it is linearly independent, and if it satisfies~\eqref{intro_eq} for $r=1$, then the only product tensors in $\spn\{x_a : a\in[n]\}$ are scalar multiples of $x_1,\dots, x_n$~\cite{1751-8121-48-4-045303}. When ${r=n}$, it holds that $\sum_{a \in [n]} x_a$ constitutes a unique tensor rank decomposition, by Kruskal's theorem. It is natural to ask what happens for $r \in \{0,1,\dots, n\}$. In Section~\ref{lowrank} we use our splitting theorem to prove that when the inequality~\eqref{intro_eq} holds, the only rank ${\leq r}$ tensors in ${\spn\{x_a : a\in[n]\}}$ are those that can be written (uniquely) as a linear combination of ${\leq r}$ elements of $\{x_a : a\in[n]\}$, which interpolates between Kruskal's theorem for $r=n$, and the results of~\cite{1751-8121-48-4-045303} for $r \in \{0,1\}$. We generalize our interpolating statement in a similar manner to our generalization of Kruskal's theorem (Theorem~\ref{hakyegen}). We also interpolate to weaker notions of uniqueness, which are explained further at the end of this introduction. We remark that the ${m=2, r=0}$ case of a result in this section was proven by Pierpaola Santarsiero in unpublished work, using a different proof technique.

%~\cite{Ballico:2018aa,ballico2020linearly}.
%Under the weakest notion of uniqueness, Theorem~\ref{hakyegen} converges to the consequence of our splitting theorem mentioned above.
%As a corollary, we strengthen recent results in \cite{Ballico:2018aa,ballico2020linearly} on circuits of product tensors.

The interpolating statement described in the previous paragraph immediately implies the following lower bound on tensor rank:
\begin{align}
\rank\bigg[\sum_{a \in [n]} x_a\bigg]\geq \min\bigg\{n, \sum_{j=1}^m (k_j-1)+2-n\bigg\}.
\end{align}
In Section~\ref{waring_rank} we use our splitting theorem to improve this bound. Namely, provided that the k-ranks are sufficiently balanced, we prove that two of the k-ranks $k_i,k_j$ appearing in this bound can be replaced by standard ranks $d_i,d_j$, improving this bound when the ranks and k-ranks are not equal. Our improved bound specializes to Sylvester's matrix rank inequality when $m=2$ \cite{horn2013matrix}. In Section~\ref{sharp_tensor_bound} we prove that our improved bound is sharp in a wide parameter regime.

%In Section~\ref{lowrank} we prove uniqueness results for low-rank tensors in $\spn\{x_a : a\in[n]\}$. In particular, for each $r\in \{0,1,\dots, n\}$ we give a condition for which the only rank $\leq r$ tensors in $\spn\{x_a : a\in[n]\}$ are those that can be written (uniquely) as a linear combination of $\leq r$ elements in $\{x_a : a\in[n]\}$, and generalize this statement to weaker notions of uniqueness. This statement generalizes results in~\cite{1751-8121-48-4-045303} for $r\in \{0,1\}$, and strengthens recent results in \cite{Ballico:2018aa,ballico2020linearly} on circuits of product tensors. We remark that the $m=2, r=0$ case of a result presented in this section was proven by Pierpaola Santarsiero in unpublished work, using a different proof technique.

In Section~\ref{symmetric_non-rank} we use our splitting theorem to prove uniqueness results for \textit{non-Waring rank} decompositions of symmetric tensors. (Our terminology for symmetric tensor decompositions is analogous to that of general tensor decompositions, and we refer the reader to Section~\ref{mp} for a formal introduction.)
%For a set of symmetric product tensors $\{v_a^{\otimes m} : a \in [n]\}$ with ${\dim\spn\{v_a : a \in [n]\}=d}$, an easy consequence of our splitting theorem is that if $n+r \leq m(d-1)+1$, then for any other multiset of symmetric product tensors $\{u_a^{\otimes m} : a \in [\tilde{r}]\}$ for which
%\begin{align}
%\{v_a^{\otimes m} : a \in [n]\} \cup \{u_a^{\otimes m} : a \in [\tilde{r}]\}
%\end{align}
%is connected, it must hold that $\tilde{r} \geq r+1$.
In particular, we prove a condition on a symmetric decomposition $v=\sum_{a \in [n]} \alpha_a v_a^{\otimes m}$ for which any other symmetric decomposition must contain at least $r_{\min}$ terms, where $r_{\min}$ depends on the rank and k-rank of $\{v_a : a \in [n]\}$. For $r_{\min} \leq n$, this gives a Waring rank lower bound that is contained in our lower bound described in the previous paragraph. For $r_{\min}=n+1$, this gives a uniqueness result for symmetric tensors that is contained in Theorem~\ref{k-gen}, but is stronger than Kruskal's theorem in a wide parameter regime. Our main contribution in this section is the case $r_{\min}>n+1$, which produces an even stronger statement than uniqueness: There are no symmetric decompositions of $v$ into a linear combination of fewer than $r_{\min}$ terms, aside from $v=\sum_{a \in [n]} \alpha_a v_a^{\otimes m}$ (up to trivialities). This is an example of what we call a uniqueness result for \textit{non-rank} decompositions of a tensor.

%In particular, we prove that if a multiset of symmetric product tensors $\{v_a^{\otimes m}: a \in [n]\}$ with $\dim\spn\{v_a : a \in [n]\}=d$ satisfies ${\krank(v_a : a \in [n])>2}$ and $n+r \leq 2d+m-3$, then for any multiset of non-zero scalars $\{\alpha_a: a \in [n]\}$, the decomposition $v=\sum_{a \in [n]} \alpha_a v_a^{\otimes m}$ satisfies the following property: Any other decomposition of $v$ must contain at least $r+1$ symmetric product tensors (up to trivialities, see Corollary~\ref{waring_uniqueness}). In Section~\ref{waring_non-rank_sharp} we prove that this result is optimal in the sense that the inequality $n+r \leq 2d+m-3$ cannot be weakened.
%
%When $r=n$ in the above statement, we obtain a uniqueness criterion for Waring rank decompositions: Whenever $\krank(v_a : a \in [n])>2$ and ${2n \leq 2d+m-3}$, it holds that $\sum_{a \in [n]} \alpha_a v_a^{\otimes m}$ constitutes a unique Waring rank decomposition. This improves Kruskal's theorem (and is independent of our generalization of Kruskal's theorem) whenever ${2d > m(k-2)+4}$, where $k=\krank(v_a : a \in [n])$. When $r > n$, we obtain a stronger uniqueness result: If $n+r \leq 2d+m-3$, then $v=\sum_{a \in [n]} \alpha_a v_a^{\otimes m}$ is the only decomposition of $v$ into $\leq r$ symmetric product tensors.

In Section~\ref{non-rank} we prove further uniqueness results for non-rank decompositions of (possibly non-symmetric) tensors. In particular, we give conditions on a multiset of product tensors $\{x_a:a\in[n]\}$ for which whenever $\sum_{a \in [n]} x_a=\sum_{a \in [r]} y_a$ for some $r>n$ and multiset of product tensors $\{y_a : a \in [r]\}$, there exist subsets $R\subseteq [n]$, $Q\subseteq [r]$ such that $\abs{Q}=\abs{R}= q$ for some fixed positive integer $q$, and $\{x_a : a\in Q\}=\{y_a : a \in R\}$. In contrast to our non-rank uniqueness results of Section~\ref{symmetric_non-rank}, which apply only to symmetric decompositions of symmetric tensors, the results of this subsection apply to arbitrary tensor decompositions.

In Section~\ref{non-rank_applications} we identify two potential applications of our uniqueness results for non-rank decompositions: First, they allow us to define a natural hierarchy of tensors in terms of ``how unique" their decompositions are. Second, any uniqueness result for non-rank decompositions can be turned around to produce a result in the more standard setting, in which one starts with a decomposition into $n$ terms, and wants to control the possible decompositions into fewer than $n$ terms.

From the proof sketch of our generalization of Kruskal's theorem that appears at the end of the previous subsection, it is easy to surmise that if $\sum_{a \in [n]} x_a=\sum_{a \in [r]} y_a$, and $2n \leq \sum_{j=1}^m (d_j^{[n]}-1)+1$, then there exist non-trvial subsets $Q \subseteq [n]$ and $R\subseteq [r]$ for which $\sum_{a \in Q} x_a =\sum_{a \in R} y_a$. This conclusion can be viewed as an extremely weakened form of uniqueness, and it is natural to ask what statements can be made for notions of uniqueness in between the standard one and this weakened one. We answer this question in Sections~\ref{lowrank} and~\ref{non-rank}.

We say that a set of non-zero vectors forms a \textit{circuit} if it is linearly dependent and any proper subset is linearly independent. As a special case of our splitting theorem, in Corollary~\ref{linincor} we obtain an upper bound on the number of subsystems $j \in [m]$ for which a {circuit} of product tensors can have $d_j \geq 2$. This improves recent bounds obtained in~\cite{Ballico:2018aa,ballico2020linearly}, and is sharp.

\section{Acknowledgments}
%-----------------------------------------------------------------------

BL thanks Edoardo Ballico, Luca Chiantini, Matthias Christandl, Harm Derksen, Dragomir {\DJ}okovi{\'{c}}, Ignat Domanov, Timothy Duff, Joshua A. Grochow, Nathaniel Johnston, Joseph M. Landsberg, Lieven De Lathauwer, Chi-Kwong Li, Daniel Puzzuoli, Pierpaola Santarsiero, William Slofstra, Hans De Sterck, and John Watrous for helpful discussions and comments on drafts of this manuscript. BL thanks Luca Chiantini and Pierpaola Santarsiero for helpful feedback on Sections~\ref{waring_rank} and~\ref{symmetric_non-rank}. In previous work \cite{tensor}, BL conjectured a (slightly weaker) ``non-minimal version" of Theorem~\ref{conjecture}. BL thanks Harm Derksen for suggesting the splitting version that appears here. BL thanks Dragomir {\DJ}okovi{\'{c}} for first suggesting a connection to Kruskal's theorem, and for suggesting that these results might hold over an arbitrary field. BL thanks Joshua A. Grochow for first asking about uniqueness results for non-rank decompositions, which inspired our results in Sections~\ref{non-rank} and~\ref{symmetric_non-rank}.

%-----------------------------------------------------------------------
\section{Mathematical preliminaries}\label{mp}
%-----------------------------------------------------------------------

Here we review some mathematical background for this work that was not covered in the introduction. For vector spaces $\V_1, \dots, \V_m$ over a field $\field$, we use $\pro{\V_1 : \dots : \V_m}$ to denote the set of (non-zero) product tensors in $\V_1\otimes  \dots \otimes \V_m$. This set forms an algebraic variety given by the affine cone over the {\it Segre variety} $\setft{Seg}(\mathbb{P} \V_1 \times \dots \times \mathbb{P} \V_m)$, with the point $0$ removed. We use symbols like $a, b$ to index tensors, and symbols like $i, j$ to index subsystems. For vector spaces $\V$ and $\W$, let $\Lin(\V,\W)$ denote the space of linear maps from $\V$ to $\W$. We use the shorthand $\Lin(\V)=\Lin(\V,\V)$. For a vector space $\V$ of dimension $d$, let $\{e_1,\dots, e_d\}$ be a standard basis for $\V$.

For a product tensor $z \in \pro{\V_1 : \dots : \V_m}$, the vectors ${z_{j} \in \V_j}$ for which ${z=z_1 \otimes \dots \otimes z_m}$ are uniquely defined up to scalar multiples $\alpha_{1} z_1, \dots, \alpha_m z_m$ such that $\alpha_1\cdots \alpha_m=1$. For positive integers $n$ and $m$, we frequently define multisets of product tensors
\begin{align}
\{ x_a : a \in [n]\} \subseteq \pro{\V_1 : \dots : \V_m}
\end{align}
without explicitly defining corresponding vectors $\{x_{a,j}\}$ such that
\begin{align}
x_a=x_{a,1}\otimes \dots \otimes x_{a,m}
\end{align}
for all $a \in [n]$. In this case, we implicitly fix some such vectors, and refer to them without further introduction.
% We also implicitly define the linear maps
%\begin{align}
%X_j=(x_{1,j}, \dots, x_{n,j}) \in \Lin(\field^n,\V_j)
%\end{align}
%for each $j \in [m]$.

 We use the notation
\begin{align}
x_{a, \hat{j}}&= x_{a,1} \otimes \dots \otimes x_{a,j-1}\otimes x_{a, j+1} \otimes \dots \otimes x_{a,m},\\
\V_{\hat{j}}&=\V_1 \otimes \dots \otimes \V_{j-1}\otimes \V_{j+1}\otimes\dots \otimes \V_{m},
\end{align}
so $x_{a,\hat{j}}\in \V_{\hat{j}}$. Note that $\V_1\otimes \dots \otimes \V_m$ is naturally isomorphic to $\Lin(\V_j^*,\V_{\hat{j}})$ for any $j \in [m]$, where $\V_j^*$ is the dual vector space to $\V_j$. The rank of a tensor in $\V_1 \otimes \V_2$ is equal to the rank of the corresponding linear operator in $\Lin(\V_1^*, \V_2)$. We denote the rank of a tensor $v\in \V$, viewed as an element of $\Lin(\V_j^*,\V_{\hat{j}})$, by $\rank_j(v)$. The \textit{flattening rank} of $v$ is defined as $\max\{\rank_1(v),\dots , \rank_m(v)\}$. Note that the tensor rank of $v$ is lower bounded by the flattening rank of $v$.
%denoted by $\setft{flat}(v)$, is defined as $\setft{flat}(v)=\max\{\rank_1(v),\dots , \rank_m(v)\}$. Note that $\setft{flat}(v)\leq \rank(v)$.

We write $S \cup T$ to denote the union of two sets $S$ and $T$. If $S$ and $T$ happen to be disjoint, we often write $S \sqcup T$ instead to remind the reader of this fact. For a positive integer $t$, we say that a collection of subsets $S_1,\dots, S_t \subseteq T$ \textit{partitions} $T$ if $S_p \cap S_q =\{\}$ for all $p\neq q \in [t]$, and $S_1 \sqcup \dots \sqcup S_t=T$.

For a multiset of non-zero vectors $E=\{v_1,\dots, v_n\}\subseteq \V$, a \textit{connected component} of $E$ is an inclusion-maximal connected subset of $E$. Any multiset of non-zero vectors $E$ can be (uniquely, up to reordering) partitioned into disjoint connected components ${T_1 \sqcup \dots \sqcup T_t=E}$~\cite[Proposition 4.1.2]{oxley2006matroid}. Observe that
\begin{align}
\spn(E)=\bigoplus_{i\in [t]} \spn(T_i),
\end{align}
and note that $S \subseteq E$ separates $E$ if and only if
\begin{align}
\dim \spn \{ v_1,\dots, v_n\} = \dim \spn \{v_a : a \in S\} + \dim\spn\{v_a : a \in S^c\}
\end{align}
if and only if
\begin{align}
\spn \{v_a : a \in S\} \cap \spn\{v_a : a \in S^c\}= \{0\}
\end{align}
(see~\cite[Proposition 4.2.1]{oxley2006matroid}).

In the remainder of this section, we formally introduce symmetric tensors and symmetric tensor decompositions, which are natural analogues of tensors and tensor decompositions. For a positive integer $m \geq 2$ and a vector space $\W$ over a field $\field$ with $\setft{Char}(\field)>m$ or $\Char(\field)=0$, we say that a tensor $v \in \W^{\otimes m}$ is \textit{symmetric} if it is invariant under permutations of the subsystems. The \textit{Waring rank} of a symmetric tensor $v$, denoted by $\setft{WaringRank}(v)$, is the minimum number $n$ for which $v$ is equal to a linear combination of $n$ symmetric product tensors. A decomposition of $v$ into a linear combination of $\setft{WaringRank}(v)$ symmetric product tensors is called a \textit{Waring rank decomposition} of $v$. A decomposition of $v$ into a linear combination of symmetric product tensors (not necessarily of minimum number) is known simply as a \textit{symmetric decomposition} of $v$.

A symmetric decomposition of $v$
\begin{align}\label{Waring_one}
v = \sum_{a \in [n]} \alpha_a v_a^{\otimes m}
\end{align}
is said to be the \textit{unique Waring rank decomposition} of $v$ if for any non-negative integer $r \leq n$, multiset of non-zero vectors $\{u_a : a \in [r]\} \subseteq \W \setminus \{0\}$, and non-zero scalars $\{\beta_a : a \in [r]\} \subseteq \field^{\times}$ for which
\begin{align}\label{Waring_two}
v = \sum_{a \in [r]} \beta_a u_a^{\otimes m},
\end{align}
it holds that $r=n$ and
\begin{align}
\{\alpha_a v_a^{\otimes m} : a \in [n]\} = \{\beta_a u_a^{\otimes m} : a \in [n]\}.
\end{align}
More generally, for a positive integer $\tilde{n} \geq n$, we say that the symmetric decomposition~\eqref{Waring_one} is the \textit{unique symmetric decomposition of $v$ into at most $\tilde{n}$ terms} if for any $r \leq \tilde{n}$ and symmetric decomposition~\eqref{Waring_two}, either
\begin{align}
{\krank(u_a : a \in [r])=1},
\end{align}
or $r=n$ and
\begin{align}
\{\alpha_a v_a^{\otimes m} : a \in [n]\} = \{\beta_a u_a^{\otimes m} : a \in [n]\}.
\end{align}
Note that~\eqref{Waring_one} is the unique Waring rank decomposition of $v$ if and only if it is the unique symmetric decomposition of $v$ into at most $n$ terms. We refer to results that certify uniqueness of a symmetric decomposition into at most $\tilde{n}>n$ terms as \textit{uniqueness results for non-Waring rank decompositions}. We present such results in Section~\ref{symmetric_non-rank}.

Our assumption that $\setft{Char}(\field)>m$ or $\Char(\field)=0$ in the symmetric case ensures that the symmetric subspace is isomorphic to the space of homogeneous polynomials over $\field$ of degree $m$ in $\dim(\W)$ variables, and that every symmetric tensor has finite Waring rank (see e.g.~\cite[Appendix~A]{iarrobino1999power} and~\cite[Section 2.6.4]{landsberg2012tensors}).

%-----------------------------------------------------------------------------%
\section{Proving the splitting theorem}

In this section, we prove Theorem~\ref{conjecture}. We first observe the following basic fact.

\begin{prop}\label{subsystem_connected_lemma}
Let $n \geq 2$ be an integer, let $\V=\V_1\otimes \V_2$ be a vector space over a field $\field$, and let
\begin{align}
E=\{x_a \otimes y_a: a \in[n] \} \subseteq \pro{\V_1 : \V_2}
\end{align}
be a multiset of product tensors. If $E$ is connected, then $\{x_a : a \in [n]\}$ and $\{y_a : a \in [n]\}$ are both connected.
\end{prop}
\begin{proof}
Suppose toward contradiction that $E$ is connected and $\{x_a : a \in [n]\}$ splits, i.e.
\begin{align}\label{simple_lemma_equation}
\dim\spn\{x_a: a \in [n]\}=\spn\{x_a: a \in S\} \oplus \spn\{x_a:a \in S^c\}
\end{align}
for some non-empty proper subset $S \subseteq [n]$. Since $E$ is connected, there exists a non-zero vector
\begin{align}
v \in \spn\{x_a \otimes y_a: a \in S\} \cap \spn\{x_a \otimes y_a: a \in S^c\}.
\end{align}
Let $f \in \V_2^*$ be any linear functional such that $(\I \otimes f) v \neq 0.$ Then $(\I \otimes f) v$ is a non-zero element of
\begin{align}
\spn\{x_a: a \in S\} \cap \spn\{x_a:a \in S^c\},
\end{align}
contradicting~\eqref{simple_lemma_equation}. The result is obviously symmetric under permutation of $\V_1$ and $\V_2$.
\end{proof}

It is not difficult to see that Theorem~\ref{conjecture} follows directly from the $m=2$ case of Theorem~\ref{conjecture}, Proposition~\ref{subsystem_connected_lemma}, and an inductive argument (we omit this proof). We therefore need only prove the $m=2$ case of Theorem~\ref{conjecture}, which we now explicitly state for clarity.

\begin{theorem}[$m=2$ case of Theorem~\ref{conjecture}]\label{bipartite}
Let $n \geq 2$ be an integer, let $\V=\V_1\otimes \V_2$ be a vector space over a field $\field$, and let
\begin{align}
E=\{x_a \otimes y_a: a \in[n] \} \subseteq \pro{\V_1 : \V_2}
\end{align}
be a multiset of product tensors. Let
\begin{align}
d_1=\dim\spn \{ x_a: a \in [n]\}
\end{align}
and
\begin{align}
d_2=\dim\spn \{ y_a: a \in [n]\}.
\end{align}
If $E$ is connected, then $\dim\spn(E)\geq d_1+d_2-1$.
\end{theorem}

To prove Theorem~\ref{bipartite}, we require a matroid-theoretic construction called the \textit{ear decomposition} of a connected matroid (see, e.g. \cite{Coullard:1996aa}). For completeness, we review the construction here. We refer the reader to \cite{oxley2006matroid} for the basic matroid-theoretic arguments used in this proof.

\begin{lemma}[Ear decomposition]\label{ear}
Let $n \geq 2$ be an integer, let $\V$ be a vector space over a field $\field$, and let $E=\{v_1,\dots, v_n\} \subseteq \V\setminus\{0\}$ be a multiset of non-zero vectors. If $E$ is connected, then there exists a collection of circuits $C_1, \dots, C_t \subseteq E$ such that
\begin{align}
E=C_1 \cup C_2 \cup \dots \cup C_t,
\end{align}
and for each $p \in [t]$, the multisets $C_p$ and $E_p := C_1 \cup \dots \cup C_p$ satisfy the following two properties:
\begin{enumerate}
\item $C_p \cap E_{p-1} \neq \{\}$
\item $\dim\spn(E_p)-\dim\spn(E_{p-1})=|E_p \setminus E_{p-1}|-1$
\end{enumerate}
\end{lemma}
\begin{proof}
Let $C_1 \subseteq E$ be an arbitrary circuit, which must exist because $E$ is non-empty and connected, and assume by induction that $C_1,\dots, C_p$ have already been constructed to satisfy properties 1 and 2. Let $B \subseteq E_p$ be a basis for $\spn(E_p)$, and choose vectors $u_1, u_2, \dots \in E \setminus E_p$ sequentially such that at each step $q$, $\{u_1,\dots, u_q\}$ is linearly independent. Terminate when
\begin{align}\label{ear_equation}
\dim\spn\{B \cup \{u_1,\dots, u_q\}\}=\abs{B}+q-1.
\end{align}
Note that this process must terminate, otherwise $E$ would split. Fixing $q$ to be the terminating step of this process, note that if $u_q$ is removed from $B \cup \{u_1,\dots, u_q\}$, then the resulting multiset is linearly independent, so $B \cup \{u_1,\dots, u_q\}$ contains a unique circuit containing $u_q$. Call this circuit $C_{p+1}$, and observe that properties 1 and 2 hold for ${E_{p+1}:=C_1 \cup \dots \cup C_{p+1}}$. The lemma follows by repeating this process until the circuits cover $E$.
\end{proof}

Now we prove Theorem~\ref{bipartite}.
\begin{proof}[Proof of Theorem~\ref{bipartite}]
For a subset $S \subseteq [n]$, let
\begin{align}
d^S&=\dim\spn\{x_a \otimes y_a : a \in S\},\\
d_1^S&=\dim\spn\{x_a : a \in S\},\\
d_2^S&=\dim\spn\{y_a: a \in S\}.
\end{align}
In a slight change of notation from Lemma~\ref{ear}, let $C_1,\dots, C_t \subseteq [n]$ be the index sets corresponding to an ear decomposition of $E$, and let $E_p=C_1 \cup \dots \cup C_p \subseteq [n]$ for each $p \in [t]$. The theorem follows from the following two claims
\begin{samepage}\begin{claim}\label{simple_claim}
\leavevmode\vspace{-\dimexpr\baselineskip + \topsep + \parsep} \begin{center}$d^{E_1} \geq d^{E_1}_1+d^{E_1}_2-1.$\end{center}
\end{claim}\end{samepage}
\begin{claim}\label{hard_claim}
For each $p \in \{2,\dots, t\}$,
\begin{align}
\abs{E_p \setminus E_{p-1}}-1 \geq d_1^{E_p}-d_1^{E_{p-1}}+d_2^{E_p}-d_2^{E_{p-1}}.
\end{align}
\end{claim}
Before proving these claims, let us first use them to complete the proof.
Note that
\begin{align}
d^{E_2}&=d^{E_1}+\abs{E_2 \setminus E_1}-1\\
	&\geq d^{E_1}_1+d^{E_1}_2-1+\abs{E_2 \setminus E_1}-1\\
	&\geq d^{E_2}_1+d^{E_2}_2-1.
\end{align}
The first line is a property of the ear decomposition, the second line follows from Claim~\ref{simple_claim}, and the third line follows from Claim~\ref{hard_claim}. So Claim~\ref{simple_claim} holds with $E_1$ replaced with $E_2$. Repeating this process inductively gives $d^{[n]} \geq d^{[n]}_1+d^{[n]}_2-1$, which is what we wanted to prove. This completes the proof, modulo proving the claims.
\begin{proof}[Proof of Claim~\ref{simple_claim}]
\renewcommand\qedsymbol{$\triangle$}
By permuting $[n]$, we may assume that $C_1=[q]$ for some $q \in [n]$, and that $\{ x_a : a \in [d_1^{[q]}]\}$ is a basis for $\spn\{x_a : a \in [q]\}$. Let $s=d_1^{[q]}$.

Suppose that there exists $b \in [s]$ such that $y_b \notin \spn \{y_a : a \in [q]\setminus [s]\}$. Let $f \in \V_1^*$, $g \in \V_2^*$ be linear functionals such that $f(x_b)=g(y_b)=1$, $f(x_a)=0$ for all $a \in [s] \setminus \{b\}$, and $g(y_a)=0$ for all $a \in [q]\setminus [s]$. So
\begin{align}
(f \otimes g)(x_a \otimes y_a)=
\begin{cases}
1, & a=b\\
0, & a \neq b
\end{cases}.
\end{align}
It follows that $x_b \otimes y_b \notin \spn \{x_a \otimes y_a : a \in [q] \setminus \{b\}\}$, contradicting the fact that $C_1$ indexes a circuit. So $\{y_a : a \in [s]\} \subseteq \spn\{y_a : a\in [q]\setminus [s]\}$, which implies
\begin{align}
d_2^{C_1} &\leq q-s\\
		&=d^{C_1}+1-d_1^{C_1},
\end{align}
completing the proof.
\end{proof}
Now we prove Claim~\ref{hard_claim}.
\begin{proof}[Proof of Claim~\ref{hard_claim}]
\renewcommand\qedsymbol{$\triangle$}
Let $B\subseteq E_{p-1}$ be such that $\{x_a : a \in B\}$ is a basis for ${\spn \{x_a : a \in E_{p-1}\}}$. By permuting $[n]$, we may assume that $E_p \setminus E_{p-1}=[q]$ for some $q \in [n]$, and that ${B \cup \{x_a : a \in [s]\}}$ is a basis for $\spn\{x_a : a \in E_p\}$, where $s=d^{E_p}-d^{E_{p-1}}$. If there exists $b \in [s]$ for which $y_b \notin \spn \{y_a : a \in [q]\setminus[s]\}$, then, as in the proof of Claim~\ref{simple_claim},
\begin{align}
x_b \otimes y_b \notin \spn \{x_a \otimes y_a : a \in E_p \setminus \{b\}\}.
\end{align}
But this contradicts connectedness of $E$, a contradiction. It follows that $d_2^{E_p \setminus E_{p-1}}\leq q-s$, so
\begin{align}
d_2^{E_p}-d_2^{E_{p-1}}&\leq d_2^{C_p}-d_2^{C_p \cap E_{p-1}}\\
					&\leq d_2^{C_p \setminus (C_p \cap E_{p-1})}-1\\
					&= d_2^{E_p \setminus E_{p-1}}-1\\
					&\leq q-s-1\\
					&=\abs{E_p \setminus E_{p-1}}-(d_1^{E_p}-d_1^{E_{p-1}})-1.
\end{align}
The first line is easy to verify (in matroid-theoretic terms, this is submodularity of the rank function). The second line follows from the fact that $\{y_a : a \in C_p\}$ is connected. The third line is obvious, the fourth line we proved above, and the fifth line follows from our definitions. This completes the proof.
\end{proof}
The proofs of Claims~\ref{simple_claim} and~\ref{hard_claim} complete the proof of the theorem.
\end{proof}

%-----------------------------------------------------------------------
\section{Using our splitting theorem to generalize Kruskal's theorem}\label{k-gen_proof}
%-----------------------------------------------------------------------
In this section we use our splitting theorem (Theorem~\ref{conjecture}) to prove our generalization of Kruskal's theorem (Theorem~\ref{k-gen}). We then introduce a reshaped version of Theorem~\ref{k-gen}, which has many more degrees of freedom than the standard reshaping of Kruskal's theorem.

To prove Theorem~\ref{k-gen}, we first observe the following useful corollary to our splitting theorem.
\begin{cor}\label{original_conjecture}
Let $n \geq 2$ and $m\geq 2$ be integers, let $\V=\V_1\otimes \dots \otimes \V_m$ be a vector space over a field $\field$, let
\begin{align}
E=\{x_a: a \in[n] \} \subseteq \pro{\V_1:\dots : \V_m}
\end{align}
be a multiset of product tensors, and for each $j \in [m]$, let
\begin{align}
d_j=\dim\spn \{ x_{a,j}: a \in [n]\}.
\end{align}
If $n \leq \sum_{j=1}^m (d_j-1)+1$, then $E$ splits.
\end{cor}
\begin{proof}
If $E$ is linearly independent, then it obviously splits. Otherwise,
\begin{align}
{\dim \spn (E) \leq n-1},
\end{align}
and the result follows immediately from our splitting theorem.
\end{proof}

Now we use this corollary to prove our generalization of Kruskal's theorem.

\begin{proof}[Proof of Theorem~\ref{k-gen}]
Let $x_a=x_{a,1}\otimes \dots \otimes x_{a,m}$ for each $a \in [n]$, and suppose that ${\sum_{a \in [n]} x_a = \sum_{a \in [r]} y_a}$ for some non-negative integer $r \leq n$ and multiset of product tensors $\{y_a: a \in [r]\}\subseteq \pro{\V_1:\dots : \V_m}$. For notational convenience, for each $a \in [r]$ let $x_{n+a}=-y_a$, so that $\sum_{a\in [n+r]} x_a =0$. Let $T_1 \sqcup \dots \sqcup T_t=[n+r]$ be the index sets of the connected components of $\{x_a : a \in [n+r]\}$. Since $\sum_{a\in [n+r]} x_a =0$, it follows that $\sum_{a \in T_p} x_a =0$ for all $p \in [t]$, so $\abs{T_p} \geq 2$ for all $p \in [t]$.

For each $p \in [t]$, if
\begin{align}\label{k-gen_proof_inequality}
\bigabs{T_p \cap [n]} \geq \bigabs{T_p \cap [n+r]\setminus [n]},
\end{align}
then it must hold that
\begin{align}\label{k-gen_proof_equality}
\bigabs{T_p \cap [n]} = \bigabs{T_p \cap [n+r]\setminus [n]}=1,
\end{align}
otherwise $\{x_a : a \in T_p\}$ would split by Corollary~\ref{original_conjecture}, a contradiction. Since $r \leq n$ and the inequality~\eqref{k-gen_proof_inequality} can never be strict, it follows that $r=n$ and~\eqref{k-gen_proof_equality} holds for all $p \in [t]$. This completes the proof.
%
%
%
% If $\abs{T_p}=2$ for all $p \in [t]$, then each $T_p$ must contain one element of $[n]$ and one element of $[n+\tilde{n}]\setminus [n]$, since no two elements of $\{x_a : a \in [n]\}$ can sum to zero. In particular, it must hold that $\tilde{n}=n$ and $\{x_a : a \in [n]\}=\{y_a : a \in [n]\}$. So it suffices to prove that $\abs{T_p}=2$ for all $p \in [t]$.
%
%If not every connected component has size two, then there exists $p \in [t]$ for which
%\begin{align}
%\abs{T_p \cap [n]} \geq \max\big\{2,\bigabs{T_p\cap [n+\tilde{n}]\setminus [n]}\big\},
%\end{align}
%but this implies, by Corollary~\ref{original_conjecture}, that $T_p$ splits, a contradiction. This completes the proof.
\end{proof}

For $m \geq 4$, both Kruskal's theorem and our Theorem~\ref{k-gen} can be ``reshaped'' by regarding multiple subsystems as a single subsystem, to give potentially stronger uniqueness criteria. It is worth noting that the reshaped version of Theorem~\ref{k-gen} has quite a different flavour from the reshaped version of Kruskal's theorem; in particular, there are many more degrees of freedom to choose from. We omit the proof of the following reshaped version of Theorem~\ref{k-gen}, because it is similar to the proof of Theorem~\ref{k-gen}.

\begin{theorem}[Reshaped generalization of Kruskal's theorem]\label{reshaped_k-gen}
Let $n \geq 2$ and $m \geq 3$ be integers, let $\V=\V_1\otimes \dots \otimes \V_m$ be a vector space over a field $\field$, and let
\begin{align}
\{x_a: a \in [n]\} \subseteq \pro{\V_1 : \dots : \V_m}
\end{align}
be a multiset of product tensors. For each $S \subseteq [n]$ and $J \subseteq [m]$, let
\begin{align}
d_J^S=\dim\spn\Big\{\bigotimes_{j \in J} x_{a,j} : a \in S\Big\}.
\end{align}
If for every subset $S \subseteq [n]$ with $2 \leq \abs{S} \leq n$ there exists a partition $J_1 \sqcup \dots \sqcup J_t =[m]$ (which may depend on $S$) such that $2 \abs{S} \leq \sum_{i \in [t]} (d_{J_i}^S-1)+1$, then $\sum_{a \in [n]} x_a$ constitutes a unique tensor rank decomposition.
\end{theorem}

It is instructive to compare Theorem~\ref{reshaped_k-gen} to the standard reshaping of Kruskal's theorem:
\begin{theorem}[Reshaped Kruskal's theorem]\label{reshaped_kruskal}
Let $n \geq 2$ and $m \geq 3$ be integers, let ${\V=\V_1\otimes \cdots \otimes \V_m}$ be a vector space over a field $\field$, and let
\begin{align}
\{x_a: a \in [n] \}\subseteq \pro{\V_1 : \dots : \V_m}
\end{align}
be a multiset of product tensors. For each $J \subseteq [m]$, let
\begin{align}
k_J = \krank( \bigotimes_{j \in J} x_{a,j} : a \in [n]).
\end{align}
If there exists a partition of $[m]$ into three disjoint subsets $J \sqcup K \sqcup L=[m]$ such that ${2n \leq k_J+k_K+k_L-2},$ then $\sum_{a \in [n]} x_a$ constitutes a unique tensor rank decomposition.
\end{theorem}
Theorem~\ref{reshaped_kruskal} clearly follows from our Theorem~\ref{reshaped_k-gen}. In Theorem~\ref{reshaped_kruskal}, one could of course consider more general partitions of $[m]$ into more than three subsets, but since the k-rank satisfies $k_{J \cup K} \geq \min\{n,k_{J}+k_K-1\}$ for any disjoint subsets $J,K \subseteq [m]$ (See Lemma~1 in~\cite{sidiropoulos2000uniqueness}), it suffices to consider tripartitions $J \sqcup K \sqcup L=[m]$. In contrast, it is not clear that one can restrict to tripartitions in Theorem~\ref{reshaped_k-gen}. There is another major difference between these two theorems: In Theorem~\ref{reshaped_kruskal}, one chooses a single partition of $[m]$, whereas in Theorem~\ref{reshaped_k-gen}, one is free to choose a different partition of $[m]$ for every $S$.

We remark that many other statements in this work (for example, the splitting theorem itself) can be reshaped similarly to Theorem~\ref{reshaped_k-gen}. We do not explicitly state these reshapings.

%-----------------------------------------------------------------------------%
\section{The inequality appearing in our splitting theorem cannot be weakened}\label{conjecture_kruskal}

In this section,
% we prove that the inequality $n \leq \sum_{j=1}^m (d_j-1)+1$ appearing in Conjecture~\ref{conjecture} cannot be weakened,
we find a connected multiset of product tensors $E=\{x_a : a \in [n]\}$ that satisfies ${\dim\spn(E) = \sum_{j=1}^m (d_j-1)+1}$. In fact, we prove that this multiset of product tensors forms a circuit, which is stronger than being connected. This proves that the bound in Corollary~\ref{linincor}, and the inequality $\dim\spn(E) \leq \sum_{j=1}^m (d_j-1)$ appearing in Theorem~\ref{conjecture}, cannot be weakened. The example we use is Derksen's~\cite{DERKSEN2013708}, which he used to prove that the inequality appearing in Kruskal's theorem cannot be weakened.
\begin{fact}\label{derksen}
For any field $\field$ with $\setft{Char}(\field)=0$, and positive integers $d_1,\dots, d_m$ with ${n-1=\sum_{j=1}^m (d_j-1)+1}$, there exist vector spaces $\V_1,\dots,\V_m$ over $\field$ and a multiset of product tensors $\{x_a:a \in [n]\}\subseteq \pro{\V_1 : \cdots : \V_m}$ that forms a circuit, and satisfies
\begin{align}
\dim\spn\{x_{a,j} : a \in [n]\}\geq d_j
\end{align}
for all $a \in [n]$.
\end{fact}
We note that if $d_1=\dots = d_m$, then the multiset of product tensors $\{x_a : a \in [n]\}$ can be taken to be symmetric in the sense introduced in Section~\ref{mp} (this is obvious from Derksen's construction~\cite{DERKSEN2013708}). As a result, our splitting theorem is also sharp for symmetric product tensors. We use this fact in Sections~\ref{waring_rank} and~\ref{symmetric_non-rank} to prove optimality of our results on symmetric decompositions. We remark that the assumption $\setft{Char}(\field)=0$ can be weakened, see~\cite{DERKSEN2013708}.
\begin{proof}[Proof of Fact~\ref{derksen}]
By Theorem 2 of \cite{DERKSEN2013708}, there exist vector spaces $\V_1,\dots, \V_m$ over $\field$, a positive integer $\tilde{n} \leq n$, and product tensors $\{x_a : a \in [\tilde{n}]\}\subseteq \pro{\V_1 : \cdots : \V_m}$ with k-ranks $d_j=\krank(x_{1,j},\dots, x_{\tilde{n},j})$ such that ${\sum_{a\in [\tilde{n}]} x_a=0}$. If $\tilde{n} < n$, then ${\tilde{n} \leq \sum_{j=1}^m (d_j-1)+1}$, which implies $\{x_a:a\in[\tilde{n}]\}$ is linearly independent by Corollary~\ref{hakyegen0} (or Proposition~3.1 in~{\cite{1751-8121-48-4-045303}}). But this contradicts $\sum_{a \in [\tilde{n}]} x_a=0$, so $\tilde{n}=n$. The equality $n=\sum_{j=1}^m (d_j-1)+2$ implies that $d_j \leq n-1$ for all $j \in [m]$. It follows that for any subset $S \subseteq [n]$ of size $\abs{S}=n-1$, it holds that ${\krank(x_{a,j} : a \in S) \geq d_j}$. Since $n-1= \sum_{j=1}^m (d_j-1)+1$, then by Corollary~\ref{hakyegen0},  ${\{x_a:a\in S\}}$ is linearly independent. It follows that $\{x_a : a \in [n]\}$ is a circuit.
\end{proof}

%-----------------------------------------------------------------------------%
\section{Interpolating between our generalization of Kruskal's theorem and an offshoot of our splitting theorem}\label{interpolate}
%-----------------------------------------------------------------------------%
For the entirety of this section, we fix non-negative integers $n\geq 2$ and $m\geq 2$, a vector space $\V=\V_1\otimes \dots \otimes \V_m$ over a field $\field$, and a multiset of product tensors ${\{x_a: a \in[n] \} \subseteq \pro{\V_1:\dots : \V_m}}$. For each subset $S\subseteq [n]$ and index $j \in [m]$, we define
\begin{align}
d_j^S=\dim\spn \{ x_{a,j}: a \in S\},
\end{align}
and use the shorthand $d_j= d_j^{[n]}$ for all $j \in [m]$.

As a consequence of our splitting theorem, if $n \leq \sum_{j=1}^m (d_j -1)+1$, then ${\{x_a: a \in [n]\}}$ splits (Corollary~\ref{original_conjecture}). Our generalization of Kruskal's theorem states that if $2 \abs{S} \leq \sum_{j=1}^m(d_j^S-1)+1$ for every subset $S \subseteq [n]$ with $2 \leq \abs{S} \leq n$, then $\sum_{a \in [n]} x_a$ constitutes a unique tensor rank decomposition. It is natural to ask what happens when other, similar inequalities hold. In particular, suppose that
\begin{align}\label{beyond_inequality}
{\abs{S}+\R(\abs{S}) \leq \sum_{j=1}^m (d_j^S -1)+1}
\end{align}
for all $S \subseteq [n]$ with $s+1 \leq \abs{S} \leq n$, for some $s \in [n-1]$ and function $\mathcal{R}: [n]\setminus [s] \rightarrow \mathbb{Z}$. What can be said about the tensors $v \in \spn \{x_a : a \in [n]\}$?

In this section, we use our splitting theorem to answer this question for choices of $s$ and $\R$ that produce useful results on tensor decompositions. In Section~\ref{lowrank} we prove uniqueness results for low-rank tensors in $\spn\{x_a : a \in [n]\}$. These results can be viewed as an interpolation between the two extreme choices of parameters in Corollary~\ref{original_conjecture} (where $s=n-1$ and $\R(n)=n$) and our generalization of Kruskal's theorem (where $s=1$ and $\R= \I$). We use this interpolation to extend several recent results in~\cite{1751-8121-48-4-045303, Ballico:2018aa,ballico2020linearly}. In Section~\ref{non-rank} we prove uniqueness results for non-rank decompositions of $\sum_{a \in [n]} x_a$ (i.e., decompositions into a non-minimal number of product tensors), which appear to be the first known results of this kind.

We will make use of the following terminology.
\begin{definition}\label{slsplit}
For positive integers $n$ and $r$, multisets of product tensors
\begin{align}
\{x_a : a \in [n]\},\{y_a : a \in [r]\} \subseteq \pro{\V_1 : \dots : \V_m},
\end{align}
and non-zero scalars
\begin{align}
\{\alpha_a : a\in [n]\}, \{\beta_a : a \in [r]\} \subseteq \field^\times,
\end{align}
for which
\begin{align}
\sum_{a \in [n]} \alpha_a x_a=\sum_{a \in [r]} \beta_a y_a,
\end{align}
% We say that the tuple $(\sum_{a \in [n]} x_a, \sum_{a \in [r]} y_a)$ $t$\textit{-splits} for some positive integer $t$ if there exist partitions $Q_1\sqcup \dots \sqcup Q_t = [n]$ and $R_1\sqcup \dots \sqcup R_t = [r]$ for which $Q_p \cap R_p \neq \{\}$ and $\sum_{a \in Q_p} x_a = \sum_{a \in R_p} y_a$ for all $p \in [t]$.
we say that the (ordered) pair of decompositions $(\sum_{a \in [n]} \alpha_a x_a, \sum_{a \in [r]} \beta_a y_a)$ has an $(s,l)$\textit{-subpartition} for some positive integers $s$ and $l$ if there exist pairwise disjoint subsets $Q_1,\dots, Q_l \subseteq [n]$ and pairwise disjoint subsets ${R_1,\dots, R_l \subseteq [r]}$ for which
\begin{align}
\max\{1, \abs{R_p}\}\leq \abs{Q_p}\leq s
\end{align}
and $\sum_{a \in Q_p} \alpha_a x_a = \sum_{a \in R_p} \beta_a y_a$ for all $p \in [l]$. We say that the pair $(\sum_{a \in [n]} \alpha_a x_a, \sum_{a \in [r]} \beta_a y_a)$ has an $(s,l)$\textit{-partition} if the sets $Q_1,\dots, Q_l \subseteq [n]$ and $R_1,\dots,R_l \subseteq [r]$ can be chosen to partition $[n]$ and $[r]$, respectively.

We say that the pair $(\sum_{a \in [n]} \alpha_a x_a, \sum_{a \in [r]} \beta_a y_a)$ is \textit{reducible} if there exist subsets $Q \subseteq [n]$ and $R \subseteq [r]$ for which $\abs{Q} > \abs{R}$ and $\sum_{a \in Q} \alpha_a x_a = \sum_{a \in R} \beta_a y_a$. We say that the pair is \textit{irreducible} if it is not reducible.

(Technically, the linear combinations appearing in the pair $(\sum_{a \in [n]} \alpha_a x_a, \sum_{a \in [r]} \beta_a y_a)$ should be regarded formally, so that they contain the data of the decompositions, and the linear combinations appearing elsewhere should be regarded as standard linear combinations in $\V$.)
\end{definition}
%Furthermore, suppose that $\sum_{a \in [n]} x_a = \sum_{a \in [r]} y_a$ for some integer $r \geq 0$ and set of product tensors $\{y_a : a \in [r]\}$. What can be said about the sets of vectors $\{x_a : a \in[n]\}$ and $\{y_a : a \in [r]\}$?
For brevity, we will often abuse notation and say that $\sum_{a \in [n]} \alpha_a x_a=\sum_{a \in [r]} \beta_a y_a$ has an $(s,l)$-subpartition (or is reducible) to mean that $(\sum_{a \in [n]} \alpha_a x_a,\sum_{a \in [r]} \beta_a y_a)$ has an $(s,l)$-subpartition (or is reducible). Note that the properties of $(s,l)$-subpartitions and reducibility are not symmetric with respect to permutation of the first and second decompositions. Typically, the first decomposition $\sum_{a \in [n]} \alpha_a x_a$ will be known, and the second decomposition $\sum_{a \in [r]} \beta_a y_a$ will be some unknown decomposition that we want to control.

An immediate consequence of Corollary~\ref{original_conjecture} is that if $\sum_{a \in [n]} x_a = \sum_{a \in [r]} y_a$ for some $r \leq n$, and the inequality~\eqref{beyond_inequality} holds for $s=n-1$ and $\R(n)=r$, then this pair of decompositions has an $(n-1, 1)$-subpartition (see Corollary~\ref{conjecture1} for a slight extension of this statement). By comparison, our generalization of Kruskal's theorem states that if $r \leq n$, and~\eqref{beyond_inequality} holds for $s=1$ and $\R=\I$, then $r=n$ and this pair of decompositions has a $(1,n)$-subpartition. In Section~\ref{lowrank} we prove statements on the existence of $(s,l)$-subpartitions for $r\leq n$, which interpolate between these two statements by trading stronger assumptions for stronger notions of uniqueness. In Section~\ref{non-rank} we prove a similar family of statements for $r\geq n+1$, obtaining novel uniqueness results for non-rank decompositions.

We conclude the introduction to this section by making a few notes about our definitions of $(s,l)$-subpartitions and reducibility. It may seem a bit strange at first that the inequality $\abs{R_p} \leq \abs{Q_p}$ appears in our definition of an $(s,l)$-subpartition. We have chosen to include this inequality because we typically want to reduce the number of product tensors that appear a decomposition. Our definition of reducibility captures a similar idea: If $n \leq r$ and $(\sum_{a \in [n]} \alpha_a x_a,\sum_{a \in [r]} \beta_a y_a)$ is reducible, then these decompositions can easily be combined to produce a decomposition into fewer than $n$ product tensors. (When $r \leq n$, reducibility of $(\sum_{a \in [r]} \beta_a y_a,\sum_{a \in [n]} \alpha_a x_a)$ captures a similar idea.) Assuming irreducibility will allow us to avoid certain pathological cases. Note that if $\sum_{a \in [n]} \alpha_a x_a$ is a tensor rank decomposition, then $(\sum_{a \in [n]} \alpha_a x_a,\sum_{a \in [r]} \beta_a y_a)$ is automatically irreducible.

Note that when $(\sum_{a \in [n]} \alpha_a x_a, \sum_{a \in [r]} \beta_a y_a)$ is irreducible, the existence of an $(s,l)$-subpartition is equivalent to the existence of pairwise disjoint subsets $Q_1,\dots, Q_l \subseteq [n]$ and pairwise disjoint subsets $R_1,\dots, R_l \subseteq [r]$ for which
\begin{align}
1 \leq \abs{R_p} = \abs{Q_p}\leq s
\end{align}
and $\sum_{a \in Q_p} \alpha_a x_a = \sum_{a \in R_p} \beta_a y_a$ for all $p \in [l]$. When $s=1$, these statements are equivalent even without the irreducibility assumption.
\subsection{Low-rank tensors in the span of a set of product tensors}\label{lowrank}
%-----------------------------------------------------------------------------%
In this subsection, we prove statements about low-rank tensors in $\spn\{x_a : a \in [n]\}$. Most of our results in this section are consequences of Theorem~\ref{hakyegen}, which is a somewhat complicated statement on the existence of $(s,l)$-partitions. For $s=1$, and any ${r \in \{0,1,\dots, n\}}$ we obtain a condition on $\{x_a : a \in [n]\}$ for which the only rank $\leq r$ tensors in $\spn\{x_a : a \in [n]\}$ are those that can be written (uniquely) as a linear combination of $\leq r$ elements of $\{x_a : a \in [n]\}$.  For $s=1,r=0$ we obtain a sufficient condition for linear independence of $\{x_a : a \in [n]\}$. For $s=1,r=1$ we obtain a sufficient condition for the only product tensors in $\spn\{x_a : a \in [n]\}$ to be scalar multiples of $x_1,\dots, x_n$. These generalize Proposition~3.1 and Theorem~3.2 in~{\cite{1751-8121-48-4-045303}, respectively. The case $s=1, r=n$ reproduces our generalization of Kruskal's theorem. For $s=n-1$, we strengthen recent results in~\cite{Ballico:2018aa,ballico2020linearly} on circuits of product tensors.

%We obtain several useful corollaries to this result. For any $r \in \{0,1,\dots, n\}$, we obtain sufficient conditions to certify $\rank(\sum_{a \in [n]} x_a) \geq r+1$. For $r=0,$ we give a sufficient condition for  The $r=0$ and $r=1$ results generalize . In Section~\ref{sn1section} we obtain sufficient conditions for the weakest notion of uniqueness to hold, which we use to strengthen recent results in~\cite{Ballico:2018aa,ballico2020linearly}.

%We note that a special case of Corollary~\ref{hakyegen0} was proven by Pierpaola Santarsiero in unpublished work, using a different proof technique.

Most of the statements in this subsection are consequences of the following theorem, which is complicated to state, but easy to prove with our splitting theorem.

\begin{theorem}\label{hakyegen}
Let $s \in [n-1]$, and $r\in \{0,1,\dots,n\}$ be integers. Suppose that for every subset $S \subseteq [n]$ with $s+1\leq \abs{S} \leq n,$ it holds that
\begin{align}\label{first_hakyegen_inequality}
\min\{2 \abs{S},\abs{S}+r\} \leq\sum_{j=1}^m (d_j^S-1)+1.
\end{align}
Then for any $v \in \spn \{x_a : a \in [n]\}$ with $\rank(v)\leq r$, and any decomposition ${v=\sum_{a \in [\tilde{r}]} y_a}$ of $v$ into $\tilde{r} \leq r$ product tensors $\{y_a : a \in [\tilde{r}]\} \subseteq \pro{\V_1:\dots : \V_m}$, the following holds: For any subset $S \subseteq [n]$ for which $\abs{S} \geq s+1$, and non-zero scalars $\{\alpha_a : a \in S\} \subseteq \field^\times$ for which it holds that
\begin{align}
\sum_{a \in S} \alpha_a x_a = \sum_{a \in [\tilde{r}]} y_a
\end{align}
and $( \sum_{a \in [\tilde{r}]} y_a,\sum_{a \in S} \alpha_a x_a)$ is irreducible, the pair of decompositions $(\sum_{a \in [n]} \alpha_a x_a, \sum_{a \in [\tilde{r}]} y_a)$ has an $(s,l)$-partition, for $l=\ceil{\abs{S}/s}$.
% More precisely, for some ${t\geq l}$, there exist partitions $Q_1\sqcup \dots \sqcup Q_t = S$ and ${R_1\sqcup \dots \sqcup R_t= [\tilde{r}]}$ for which $\max\{1, \abs{R_p}\}\leq \abs{Q_p} \leq s$ and $\sum_{a \in Q_p} \alpha_a x_a= \sum_{a \in R_p} y_a$ for all $p \in [t]$.
\end{theorem}

\begin{proof}
%Let $\tilde{r}=\rank(v)$, let $v=\sum_{a \in [\tilde{r}]} y_a$ be a tensor rank decomposition of $v$, and for each $a \in [\tilde{r}]$ let $x_{n+a}=-y_a$. Since $v \in \spn \{x_a : a \in [n]\}$, it holds that $\sum_{a \in R} \alpha_a x_a =0$ for some subset $R \subseteq [n+\tilde{r}]$ with $R \supseteq [n+\tilde{r}] \setminus [n]$, and non-zero scalars $\{\alpha_a : a \in R\} \subseteq \field^\times$.

For each $a \in [\tilde{r}]$, let $x_{n+a}=-y_a$, and let $E = S \cup ([n+\tilde{r}]\setminus [n]) \subseteq [n+\tilde{r}]$. Let ${T_1\sqcup \dots \sqcup T_t = E}$ be a partition of $E$ into index sets corresponding to the connected components of $\{x_a : a \in E\}$. Since $( \sum_{a \in [\tilde{r}]} y_a,\sum_{a \in S} \alpha_a x_a)$ is irreducible, it must hold that
\begin{align}
\bigabs{T_p \cap S}\geq \bigabs{T_p \cap (E \setminus S)}
\end{align}
for all $p \in [t]$, and hence
\begin{align}
\abs{T_p} \leq \min\big\{ 2 \bigabs{T_p \cap S}, \bigabs{T_p \cap S} +r\big\}.
\end{align}
If $\abs{T_p \cap S} \geq s+1,$ then $\{x_a : a \in T_p\}$ splits by~\eqref{first_hakyegen_inequality} and Corollary~\ref{original_conjecture}, a contradiction. So it must hold that ${\abs{T_p \cap S} \leq s}$ for all $p \in [t]$. It follows that $t \geq \ceil{\abs{S}/s}$ by the pigeonhole principle, and one can take $Q_p=T_p \cap S$ and
\begin{align}
R_p = \{a \in [\tilde{r}] : n+a \in T_p \cap (E \setminus S)\}
\end{align}
for all $p \in [t]$ to conclude.
\end{proof}

\subsubsection{$s=1$ case of Theorem~\ref{hakyegen}}\label{section_s1hakyegen}

The $s=1$ case of Theorem~\ref{hakyegen} gives a sufficient condition for which the only tensor rank $\leq r$ elements of $\spn\{x_a : a \in [n]\}$ are those which can be written (uniquely) as a linear combination of $\leq r$ elements of $\spn\{x_a : a \in [n]\}$. In this subsection, we state this case explicitly, and observe several consequences of this case. In particular, we observe a lower bound on tensor rank and a sufficient condition for a set of product tensors to be linearly independent.

\begin{cor}[$s=1$ case of Theorem~\ref{hakyegen}]\label{s1hakyegen}
Let $r\in \{0,1,\dots,n\}$ be an integer. Suppose that for every subset $S \subseteq [n]$ such that $2\leq \abs{S} \leq n$, it holds that
\begin{align}\label{hakyegen_inequality}
\abs{S}+\min\{\abs{S},r\} \leq\sum_{j=1}^m (d_j^S-1)+1.
\end{align}
Then any non-zero linear combination of more than $r$ elements of $\{x_a : a \in [n]\}$ has tensor rank greater than $r$, and every tensor $v \in \spn\{x_a : a \in [n]\}$ of tensor rank at most $r$ has a unique tensor rank decomposition into a linear combination of elements of $\{x_a : a \in [n]\}$.
\end{cor}

Note that a sufficient condition for the inequality~\eqref{hakyegen_inequality} to hold is that
\begin{align}
n+r \leq \sum_{j=1}^m (k_j-1)+1,
\end{align}
where $k_j=\krank(x_{1,j},\dots, x_{n,j})$ for all $j \in [m]$. This recovers Proposition~3.1 and Theorem~3.2 in~{\cite{1751-8121-48-4-045303} in the $r=0$ and $r=1$ cases, respectively, and interpolates between Kruskal's theorem and these results. For clarity, we will explicitly state the $r=0$ and $r=1$ cases of Corollary~\ref{s1hakyegen} at the end of this subsection.

\begin{proof}[Proof of Corollary~\ref{s1hakyegen}]
Let $S \subseteq [n]$ be a subset, let $\{\alpha_a : a \in S\} \subseteq \field^{\times}$ be a multiset of non-zero scalars, let $\tilde{r}=\rank[\sum_{a \in S} \alpha_a x_a]$, and let $\{y_a : a \in [\tilde{r}]\}\subseteq \pro{\V_1 : \dots : \V_m}$ be such that $\sum_{a \in S} \alpha_a x_a=\sum_{a \in [\tilde{r}]} y_a$. If $\tilde{r}\leq r$, then by the $s=1$ case of Theorem~\ref{hakyegen}, this pair of decompositions has a $(1,\abs{S})$-partition. It follows that $\abs{S}=\tilde{r}$. Hence, every linear combination of more than $r$ elements of $\{x_a : a \in [n]\}$ has tensor rank greater than $r$.

Let $v \in \spn \{x_a : a \in [n]\}$ have tensor rank $\tilde{r} \leq r$. Then ${v= \sum_{a \in Q} \alpha_a x_a}$ for some set $Q \subseteq [n]$ of size $\abs{Q}=\tilde{r}$ and non-zero scalars $\{ \alpha_a : a \in Q\}$. It follows from~\eqref{hakyegen_inequality} and Theorem~\ref{k-gen} that this is the unique tensor rank decomposition of $v$.
\end{proof}

Corollary~\ref{s1hakyegen} immediately implies the following lower bound on $\rank[\sum_{a \in [n]} x_a]$.
\begin{cor}\label{tr_lower}
If for every subset $S \subseteq [n]$ for which $2\leq \abs{S} \leq n,$ it holds that
\begin{align}\label{tr_lower_eq}
\abs{S}+\min\{\abs{S},r\} \leq\sum_{j=1}^m (d_j^S-1)+1,
\end{align} 
then $\rank[\sum_{a \in [n]} x_a]\geq r+1$.
\end{cor}

In particular, Corollary~\ref{tr_lower} implies that
\begin{align}\label{k-rank_bound}
\rank\bigg[\sum_{a \in [n]} x_a\bigg]\geq \min\bigg\{n, \sum_{j=1}^m (k_j-1)+2-n\bigg\}.
\end{align}
In Section~\ref{waring_rank} we prove that when the Kruskal ranks are sufficiently balanced, two of the k-ranks $k_i,k_j$ appearing in the bound~\eqref{k-rank_bound} can be replaced with standard ranks $d_i,d_j$ (Theorem~\ref{tensor_cor}). Our Theorem~\ref{tensor_cor} is independent of the bound in Corollary~\ref{tr_lower} (see Example~\ref{ex:independent}).

We close this subsection by stating the $r=0$ and $r=1$ cases of Corollary~\ref{s1hakyegen}, which generalize Proposition~3.1 and Theorem~3.2 in~{\cite{1751-8121-48-4-045303}, respectively.  We remark that the $m=2$ subcase of Corollary~\ref{hakyegen0} was proven by Pierpaola Santarsiero in unpublished work, using a different proof technique.
\begin{cor}[$s=1$, $r=0$ case of Theorem~\ref{hakyegen}]\label{hakyegen0}
If for every subset $S \subseteq [n]$ for which $2\leq \abs{S} \leq n,$ it holds that
\begin{align}
\abs{S} \leq\sum_{j=1}^m (d_j^S-1)+1,
\end{align}
then $\{x_a : a \in [n]\}$ is linearly independent.
\end{cor}
\begin{cor}[$s=1$, $r=1$ case of Theorem~\ref{hakyegen}]\label{hakyegen1}
If for every subset $S \subseteq [n]$ for which $2\leq \abs{S} \leq n,$ it holds that
\begin{align}
\abs{S} \leq\sum_{j=1}^m (d_j^S-1),
\end{align}
then
\begin{align}
\spn\{x_a : a \in [n]\} \cap \pro{\V_1 : \dots : \V_m}=\complex^\times x_1 \sqcup \dots \sqcup \complex^\times x_n.
\end{align}
\end{cor}

%We next state and prove what happens when $n+r\leq \sum_{j=1}^m (d_j-1)+1$, i.e., when the k-ranks are replaced by standard ranks, and no further conditions are placed on subsets $S \subseteq [n]$. We then observe an immediate consequence, Corollary~\ref{linincor}, which generalizes several recent results on linear combinations of product tensors.

%a corollary to Theorem~\ref{sep_combo} with connections to Condition~U in the study of uniqueness conditions (see Section~\ref{uniqueness_applications}), and also to the generalization of Condition~U to the case of more than three subsystems (see~\eqref{Ugeneq}).
\subsubsection{$s=n-1$ case of Theorem~\ref{hakyegen}}\label{sn1section}
In this subsection we state a slight adaptation of the $s=n-1$ case of Theorem~\ref{hakyegen}, which gives sufficient conditions for a pair of decompositions to have an $(n-1,1)$-subpartition. After stating this case, we observe that the subcase $r=1$ improves recent results in~\cite{Ballico:2018aa,ballico2020linearly} concerning circuits of product tensors. We then remark on applications of this special case in quantum information theory.

\begin{cor}[$s=n-1$ case of Theorem~\ref{hakyegen}]\label{conjecture1}
Let $r \in \{0,1,\dots, n\}$ be an integer.
If ${n+r \leq \sum_{j=1}^m (d_j-1)+1}$, then for any non-negative integer $\tilde{r} \leq r$ and multiset of product tensors $\{y_a : a \in [\tilde{r}]\}$ for which $\sum_{a \in [n]} x_a = \sum_{a \in [\tilde{r}]} y_a$, the pair of decompositions $(\sum_{a \in [n]} x_a,\sum_{a \in [\tilde{r}]} y_a)$ has an $(n-1,1)$-subpartition.

Moreover, if ${n+r \leq \sum_{j=1}^m (d_j-1)+1}$, $\tilde{r}=\rank[\sum_{a \in [n]} x_a]$, and $1\leq \tilde{r} \leq \min\{r,n-1\}$, then there exists a subset $S \subseteq [n]$ with $\tilde{r} \leq \abs{S} \leq n-1$ for which

\begin{align}\label{Ulike}
\rank\bigg[\sum_{a \in S} x_a\bigg] < \tilde{r}.
\end{align}
\end{cor}
\begin{proof}
The statement of the first paragraph is slightly different from the $s=n-1$ case of Theorem~\ref{hakyegen}, and it follows easily from Corollary~\ref{original_conjecture}. To prove the statement of the second paragraph, let $\{z_a : a \in [\tilde{r}]\}\in \pro{\V_1: \dots : \V_m}$ be any multiset of product tensors for which $\sum_{a \in [n]} x_a = \sum_{a \in [\tilde{r}]} z_a$, and let $Q \subseteq [n]$, $R \subseteq [\tilde{r}]$ be subsets for which
\begin{align}
\max\{\abs{R},1\} \leq \abs{Q} \leq n-1
\end{align}
and $\sum_{a \in Q} x_a = \sum_{a \in R} z_a$. If $\abs{R}<\abs{Q}$ and $\abs{Q} \geq \tilde{r}$, then we can take $S=Q$. If $\abs{R}<\abs{Q}$ and $\abs{Q} \leq \tilde{r}-1$, then we can take $S\subseteq [n]$ to be any subset for which $S \supseteq Q$ and $\abs{S}=\tilde{r}$. It remains to consider the case $\abs{R}=\abs{Q}$. In this case, it must hold that $\bigabs{[\tilde{r}]\setminus R} < \bigabs{[n] \setminus Q},$ so we can find $S$ using the same arguments as in the case $\abs{R}<\abs{Q}$.
\end{proof}

A special case of the $r=1$ case of Corollary~\ref{conjecture1} gives an upper bound of $n-2$ on the number of subsystems $j\in [m]$ for which a circuit of product tensors can have $d_j >1$. This bound improves those obtained in~\cite[Theorem 1.1]{ballico2020linearly} and~\cite[Lemma 4.5]{Ballico:2018aa}, and is sharp (see Section~\ref{conjecture_kruskal}).

\begin{cor}\label{linincor}
\sloppy
If $\{{x_a}: a \in [n] \}$ forms a circuit, then ${d_j >1}$ for at most $n-2$ indices $j \in [m]$.
\end{cor}
\begin{proof}
This follows immediately from Corollary~\ref{original_conjecture}, since circuits are connected. Alternatively, this follows from the second paragraph in the statement of Corollary~\ref{conjecture1}, since for any circuit it holds that $\sum_{a \in S} x_a \neq 0$ for all $S \subseteq [n]$ with $1 \leq \abs{S} \leq n-1$.
\end{proof}

As an immediate consequence of Corollary~\ref{linincor}, a sum of two product tensors is again a product tensor if and only if $d_j >1$ for at most a single subsystem index $j \in [m]$ (see Corollary~15 in \cite{tensor2}). This statement is well-known. In particular, it was used in \cite{westwick1967, johnston2011characterizing} to characterize the invertible linear operators that preserve the set of product tensors. In~\cite{lovitz2021decomposable,tensor2} the first author used this statement to study decomposable correlation matrices, and observed that it directly provides an elementary proof of a recent result in quantum information theory~\cite{PhysRevA.95.032308} (see Corollary~16 in~\cite{tensor2}).

%-----------------------------------------------------------------------
\subsection{Uniqueness results for non-rank decompositions}\label{non-rank}
%-----------------------------------------------------------------------

In this subsection we prove uniqueness results for decompositions of $\sum_{a \in [n]} x_a$ into ${r\geq n+1}$ product tensors. Namely, we provide conditions on $\{x_a : a \in [n]\}$ for which whenever $\sum_{a \in [n]} x_a = \sum_{a \in [r]} y_a$ for some multiset of product tensors $\{y_a : a \in [r]\}$, this pair of decompositions has an $(s,l)$-subpartition. In particular, for $s=1$ we obtain sufficient conditions for the existence of subsets $Q \subseteq [n]$, $R\subseteq [r]$ of size $\abs{Q}=\abs{R}=l$ for which $\{x_a : a\in Q\}=\{y_a : a \in R\}$. We refer the reader also to Section~\ref{symmetric_non-rank}, in which we prove uniqueness results on non-Waring rank decompositions of symmetric tensors, and identify applications of our non-rank uniqueness results.

In Theorem~\ref{k_arbitrary} we give sufficient conditions for which whenever $(\sum_{a \in [n]} x_a, \sum_{a \in [r]} y_a)$ is irreducible, it has an $(s,l)$-subpartition. We then observe that for $s=1$ we can drop the irreducibility assumption and obtain the result described in the previous paragraph. We then prove a modified version of Theorem~\ref{k_arbitrary}, which drops the irreducibility assumption for arbitrary $s\in [n-1]$. At the end of this subsection, we review these statements in the $s=n-1$ case.

\begin{theorem}\label{k_arbitrary}
Let $n \geq 2$, $q \in [n-1]$, $s \in [q]$, and $r$ be positive integers for which
\begin{align}\label{k_arbitrary_r_inequality}
n+1 \leq r \leq n+\Bigceil{\frac{n-q}{s}},
\end{align}
and let $l=\floor{q/s}$. If for every subset $S \subseteq [n]$ for which $s+1\leq \abs{S} \leq n,$ it holds that
\begin{align}\label{k_arbitrary_inequality}
2\abs{S}+\max\left\{0,(r-n)-\biggceil{\frac{n-q+s}{\abs{S}}}+1\right\} \leq \sum_{j=1}^m (d_j^S -1)+1,
\end{align}
then for any multiset of product tensors $\{y_a : a \in [r]\}\subseteq \pro{\V_1 : \dots : \V_m}$ for which ${\sum_{a \in [n]} x_a= \sum_{a \in [r]} y_a}$ and $(\sum_{a \in [n]} x_a,\sum_{a \in [r]} y_a)$ is irreducible, this pair of decompositions has an $(s, l)$-subpartition.
\end{theorem}

One may be concerned about whether the complicated collection of inequalities~\eqref{k_arbitrary_inequality} can ever be satisfied. The answer is yes, simply because the righthand side can depend on $m$, whereas the lefthand side does not. So for $m$ large enough, one can always find $\{x_a : a \in [n]\}$ that satisfies these inequalities. In fact, they can even be satisfied non-trivially for $m=3$, as we observe in Example~\ref{identity_example}.
%we apply the $s=1$ case of Theorem~\ref{k_arbitrary} to the \textit{identity tensor} $\sum_{a \in [n]} e_a^{\otimes 3}$.

\begin{proof}[Proof of Theorem~\ref{k_arbitrary}]
For each $a \in [r]$, let $x_{n+a}=-y_a$, and let $T_1 \sqcup \dots \sqcup T_t=[n+r]$ be the index sets of the decomposition of $\{x_a : a \in [n+r]\}$ into connected components. Note that for each $p \in [t]$, it must hold that
\begin{align}
\bigabs{T_p \cap [n+r]\setminus [n]}\geq \bigabs{T_p \cap [n]},
\end{align}
otherwise we would contradict irreducibility. For each $p \in [t]$, if
\begin{align}
\bigabs{T_p \cap [n+r]\setminus [n]}= \bigabs{T_p \cap [n]},
\end{align}
then $\bigabs{T_p \cap [n]} \leq s$, otherwise $\{ x_a : a \in T_p\}$ would split by~\eqref{k_arbitrary_inequality} and Corollary~\ref{original_conjecture}. Assume without loss of generality that
\begin{align}
\bigabs{T_1 \cap [n]}-\bigabs{T_1 \cap [n+r]\setminus [n]} &\geq \bigabs{T_2 \cap [n]}-\bigabs{T_2 \cap [n+r]\setminus [n]}\\
 &\;\; \vdots \\
 & \geq \bigabs{T_{{t}} \cap [n]}-\bigabs{T_{{t}} \cap [n+r]\setminus [n]}.
\end{align}
If
\begin{align}
\bigabs{T_{1}\cap [n]} = \bigabs{T_{1} \cap [n+r]\setminus [n]},
\end{align}
then let $\tilde{l}\in [t]$ be the largest integer for which
\begin{align}\label{equality}
\bigabs{T_{\tilde{l}}\cap [n]} = \bigabs{T_{\tilde{l}} \cap [n+r]\setminus [n]}.
\end{align}
Otherwise, let $\tilde{l}=0$. Then for all $p \in [t]\setminus [\tilde{l}]$ it holds that
\begin{align}\label{strict_inequality}
\abs{T_p \cap [n]} < \abs{T_p \cap [n+r]\setminus [n]}
\end{align}
(recall that we define $[0]=\{\}$). To complete the proof, we will show that $\tilde{l} \geq l$, for then we can take $Q_p=T_p \cap [n]$ and $R_p=T_p \cap [n+r]\setminus [n]$ for all $p \in [l]$ to conclude.

Suppose toward contradiction that $\tilde{l}< l$. We require the following two claims:
\begin{claim}\label{pigeonhole_claim}
It holds that $\tilde{l}<t$, $\bigceil{\frac{n- s \tilde{l} }{t-\tilde{l}}} \geq s+1$, and there exists $p \in [t]\setminus [\tilde{l}]$ for which
\begin{align}\label{pigeonhole}
\bigabs{T_p \cap [n]} \geq \biggceil{\frac{n- s \tilde{l} }{t-\tilde{l}}}.
\end{align} 
\end{claim}

\begin{claim}\label{inequality2_claim}
For all $p \in [t] \setminus [\tilde{l}]$, it holds that
\begin{align}\label{inequality2}
\bigabs{T_p \cap [n+r]\setminus [n]} \leq \bigabs{T_p \cap [n]} +r-n+\tilde{l} -t+1
\end{align}
\end{claim}

Before proving these claims, we first use them to complete the proof of the theorem. Let $p \in [t]\setminus [\tilde{l}]$ be as in Claim~\ref{pigeonhole_claim}. Then,
\begin{align}
\abs{T_p} &= \bigabs{T_p \cap [n]}+\bigabs{T_p \cap [n+r] \setminus [n]}\\
		&\leq 2 \bigabs{T_p \cap [n]}+ r-n+\tilde{l} -t+1\\
		&\leq 2 \bigabs{T_p \cap [n]}+ r-n - \biggceil{\frac{n-s \tilde{l}}{\bigabs{T_p \cap [n]}}}+1\\
		%&\leq 2 \bigabs{T_p \cap [n]}+ r-n+s(l-1)- \frac{n-s (l-1)}{\bigabs{T_p \cap [n]}}+1\\
		&\leq 2 \bigabs{T_p \cap [n]}+ r-n- \biggceil{\frac{n-q+s}{\bigabs{T_p \cap [n]}}}+1\\
		& \leq \sum_{j=1}^m (d_j^{T_p \cap [n]}-1)+1,
\end{align}
where the first line is obvious, the second follows from Claim~\ref{inequality2_claim}, the third follows from Claim~\ref{pigeonhole_claim}, the fourth follows from $\tilde{l}< l$, and the fifth follows from~\eqref{k_arbitrary_inequality} and the fact that $\bigabs{T_p \cap [n] } \geq s+1$. So $\{x_a : a \in T_p\}$ splits, a contradiction. This completes the proof, modulo proving the claims.

\begin{proof}[Proof of Claim~\ref{pigeonhole_claim}]
\renewcommand\qedsymbol{$\triangle$}
To prove the claim, we first observe that $n>st$. Indeed, if $n\leq st$ then
\begin{align}
r & = \sum_{p=1}^t \bigabs{T_p \cap [n+r]\setminus [n]}\\
&\geq n+t-\tilde{l} \\
&\geq n+\frac{n-q}{s}+1,
\end{align}
%\begin{align}
%r &< \left(\frac{s+1}{s}\right)(n-q+s)\\
%  &=n+\frac{n}{s}-(s+1)(q/s-1)\\
%  &\leq n+t - (s+1)\tilde{l}\\
%  &\leq n+t-\bigabs{(T_1 \sqcup \dots \sqcup T_{\tilde{l}})\cap [n]}-\tilde{l}\\
%  &=\sum_{p=\tilde{l}+1}^t (\abs{T_p \cap [n]}+1)\\
%  &\leq \sum_{p=\tilde{l}+1}^t \bigabs{T_p \cap [n+r]\setminus [n]}\\
%  &\leq r,
%  \end{align}
where the first line is obvious, the second follows from~\eqref{equality} and~\eqref{strict_inequality}, and the third follows from $n\leq st$ and $\tilde{l}<l$. This contradicts~\eqref{k_arbitrary_r_inequality}, so it must hold that $n>st$.

Note that $\tilde{l}<t$, for otherwise we would have $n \leq st$ by the fact that $\bigabs{T_p \cap [n]} \leq s$ for all $p \in [\tilde{l}]$. To verify that $\bigceil{\frac{n- s \tilde{l} }{t-\tilde{l}}} \geq s+1$, it suffices to prove $\frac{n- s \tilde{l} }{t-\tilde{l}} > s,$ which follows from $n>st$. To verify~\eqref{pigeonhole}, since $\bigabs{T_p \cap [n]} \leq s$ for all $p \in [\tilde{l}]$, by the pigeonhole principle there exists $p \in [t]\setminus [\tilde{l}]$ for which
\begin{align}
\bigabs{T_p \cap [n]} \geq \biggceil{\frac{n- s \tilde{l} }{t-\tilde{l}}}.
\end{align}
This proves the claim.
\end{proof}
\begin{proof}[Proof of Claim~\ref{inequality2_claim}]
\renewcommand\qedsymbol{$\triangle$}
Suppose toward contradiction that the inequality~\eqref{inequality2} does not hold for some $\tilde{p}\in [t] \setminus [\tilde{l}]$. Then
\begin{align}
r &= \sum_{p=1}^t \bigabs{T_p \cap [n+r]\setminus [n]}\\
& \geq \sum_{p\neq \tilde{p}} \bigabs{T_p \cap [n+r]\setminus [n]} + \bigabs{T_{\tilde{p}} \cap[n]} + (r-n)+\tilde{l}-t+2\\
%& \geq n+(t-\tilde{l}-1)+(r-n)+\tilde{l}-t+2\\
&\geq r+1,
\end{align}
where the first two lines are obvious, and the last line follows from~\eqref{equality} and~\eqref{strict_inequality}, a contradiction.
\end{proof}
The proofs of Claims~\ref{pigeonhole_claim} and~\ref{inequality2_claim} complete the proof of the theorem.
\end{proof}

\subsubsection{$s=1$ case of Theorem~\ref{k_arbitrary}}

In the $s=1$ case of Theorem~\ref{k_arbitrary}, we can drop the assumption that the pair of decompositions is irreducible. This is because the other assumptions already imply that $\sum_{a \in [n]} x_a$ constitutes a (unique) tensor rank decomposition by Theorem~\ref{k-gen}, so $\sum_{a \in [n]} x_a=\sum_{a \in [r]} y_a$ will automatically be irreducible (see the discussion at the beginning of Section~\ref{interpolate}).

\begin{cor}[$s=1$ case of Theorem~\ref{k_arbitrary}]\label{k_arbitrarys1}
Let $q \in [n-1]$ and $r$ be positive integers for which $n+1 \leq r \leq 2n-q$. If for every subset $S \subseteq [n]$ with $2\leq \abs{S} \leq n$ it holds that
\begin{align}\label{nonrank_inequality} 
2\abs{S}+\max\left\{0,(r-n)-\biggceil{\frac{n-q+1}{\abs{S}}}+1\right\} \leq \sum_{j=1}^m (d_j^S -1)+1,
\end{align}
then for any multiset of product tensors $\{y_a : a \in [r]\}\subseteq \pro{\V_1 : \dots : \V_m}$ for which $\sum_{a \in [n]} x_a= \sum_{a \in [r]} y_a$, there exist subsets $Q \subseteq [n]$ and $R \subseteq [r]$ of size $\abs{Q}=\abs{R}=q$ for which ${\{x_a : a \in Q\}=\{y_a : a \in R\}}$ (in other words, this pair of decompositions has a $(1,q)$-subpartition).
\end{cor}
It is worth noting that although the assumptions of Corollary~\ref{k_arbitrarys1} require $\sum_{a \in [n]} x_a$ to constitute a unique tensor rank decomposition, this result can also be applied to arbitrary decompositions $\sum_{a \in [n]} x_a$, provided that $\sum_{a\in S} x_a$ constitutes a unique tensor rank decomposition for some subset $S\subseteq [n]$ with $2 \leq \abs{S}\leq n$, as one can simply apply Corollary~\ref{k_arbitrarys1} to the pair of decompositions $(\sum_{a \in S} x_a , \sum_{a\in [r]} y_a- \sum_{a\in [n]\setminus S} x_a)$. It is not difficult to produce explicit examples in which Corollary~\ref{k_arbitrarys1} can be applied in this way (for instance, by modifying Example~\ref{identity_example}).

As an example, we now use Corollary~\ref{k_arbitrarys1} to prove uniqueness of non-rank decompositions of the \textit{identity tensor} $\sum_{a \in [n]} e_a^{\otimes 3}$.
\begin{example}\label{identity_example}
Let $n \geq 2$, $q \in [n-1],$ and $r$ be positive integers for which $n+1 \leq r \leq 2n-q$ and
\begin{align}\label{example_inequality}
q \leq n+1-\frac{1}{4}\left((r-n+2)^2+1\right).
\end{align}
%\begin{align}
%4(n-q+1) \geq (r-n+2)^2+1.
%\end{align}
If
\begin{align}
\sum_{a \in [n]} e_a ^{\otimes 3}=\sum_{a \in [r]} y_a
\end{align}
for some multiset of product tensors $\{y_a : a \in [r]\}\subseteq \pro{\V_1 : \V_2 : \V_3}$, then there exist subsets $Q \subseteq [n]$ and $R \subseteq [n+r]$ of sizes $\abs{Q}=\abs{R}=q$ such that ${\{x_a : a \in Q\} = {\{y_a : a \in R\}}}$. For example, if $r=n+1$ then we can take $q= n-2$ for any $n \geq 3$.
\end{example}

To verify Example~\ref{identity_example}, it suffices to show that the inequality~\eqref{nonrank_inequality} holds for all $S \subseteq [n]$ with $2 \leq \abs{S} \leq n$. This reduces to proving that
\begin{align}
\abs{S}(r-n+2-\abs{S}) - (n-q+1)<0,	
\end{align}
which occurs whenever the polynomial in $\abs{S}$ on the lefthand side has no real roots, i.e. whenever 
\begin{align}
(r-n+2)^2 \leq 4(n-q+1)-1.
\end{align}
\subsubsection{Modifying Theorem~\ref{k_arbitrary} to apply to reducible pairs of decompositions}
A drawback to Theorem~\ref{k_arbitrary} is that it only applies to irreducible pairs of decompositions. We now present a modification of this result, which can certify the existence of an $(s,l)$-subpartition even for reducible decompositions, at the cost of stricter assumptions. We defer this proof to the appendix, as it is very similar to that of Theorem~\ref{k_arbitrary}.
\begin{theorem}\label{k_arbitrary_old}
Let $q \in [n-1]$, $s \in [q]$, and $r$ be positive integers for which
\begin{align}\label{eq:k_arbitrary_old}
n+1 \leq r \leq \biggceil{\left(\frac{s+1}{s}\right)(n-q+s)}-1,
\end{align}
and let $l=\floor{q/s}$. If for every subset $S \subseteq [n]$ for which $s+1\leq \abs{S} \leq n,$ it holds that
\begin{align}
2\abs{S}+\max\left\{0,(r-n+q-s)-\biggceil{\frac{n-q+s}{\abs{S}}}+1\right\} \leq \sum_{j=1}^m (d_j^S -1)+1,
\end{align}
then for any multiset of product tensors $\{y_a : a \in [r]\}\subseteq \pro{\V_1 : \dots : \V_m}$ for which $\sum_{a \in [n]} x_a= \sum_{a \in [r]} y_a$, this pair of decompositions has an $(s,l)$-subpartition.
\end{theorem}

\subsubsection{$s=n-1$ case of Theorem~\ref{k_arbitrary_old}}

When $s=n-1$, then it necessarily holds that $r=n+1$ and $q=n-1$, and Theorem~\ref{k_arbitrary_old} simply says that if $2n+1 \leq \sum_{j=1}^m (d_j-1)+1$, then $\sum_{a \in [n]} x_a= \sum_{a \in [n+1]} y_a$ has an $(n-1,1)$-subpartition. Theorem~\ref{k_arbitrary} yields a weaker statement.
%, because our splitting theorem is sharp (See Section~\ref{conjecture_kruskal}), and as soon as $r\geq n+2$, one can find non-trivial partitions $Q_1 \sqcup Q_2=\{x_a : a\in [n]\}$ and $R_1 \sqcup R_2= \{y_a : a \in [r]\}$ for which $\abs{Q_p} < \abs{R_p}$ for all $p \in [2]$.

%-----------------------------------------------------------------------
\section{A lower bound on tensor rank}\label{waring_rank}
%-----------------------------------------------------------------------

In Section~\ref{section_s1hakyegen} we saw that for a multiset of product tensors $\{x_a : a \in [n]\}$ with k-ranks $k_j=\krank(x_{a,j} : a \in [n])$, it holds that
\begin{align}\label{k-rank_bound1}
\rank\bigg[\sum_{a \in [n]} x_a\bigg]\geq \min\bigg\{n, \sum_{j=1}^m (k_j-1)+2-n\bigg\}.
\end{align}
In this section, we prove that when the k-ranks are sufficiently balanced, two of the k-ranks $k_i,k_j$ appearing in this bound can be replaced with standard ranks $d_i,d_j$, which improves this bound when the k-ranks and ranks are not equal, and specializes to Sylvester's matrix rank inequality when $m=2$. We prove that this improved bound is independent of a different lower bound on tensor rank that we observed in Corollary~\ref{tr_lower}. We furthermore observe that this improved bound is sharp in a wide parameter regime. As a consequence, we obtain a lower bound on Waring rank, which we also prove is sharp.

%In this section, we prove lower bounds on tensor rank and Waring rank, inspired by results in~\cite{Chiantini2019}. Given a decomposition of a tensor as a sum of $n$ product tensors, it is natural to ask if properties of the decomposition can be used to lower bound its rank. We have already seen one example of such a bound in Corollary~\ref{tr_lower}. In this section, we prove sharp lower bounds on the tensor rank based on the k-ranks and standard ranks of a given decomposition, improving~\cite[Theorem~6.8 and Remark~6.14]{Chiantini2019} in the symmetric case. More generally, it is also interesting to ask what other (potentially non-rank) decompositions of a tensor can look like.

%We now state and prove our lower bound on tensor rank, and observe a consequent lower bound on Waring rank of symmetric tensors. In Section~\ref{sharp_tensor_bound} we observe that neither of these bounds can be weakened.

%\subsection{Sharp lower bounds on tensor rank and Waring rank}\label{tensor_rank_lower_bound}

\begin{theorem}[Tensor rank lower bound]\label{tensor_cor}
Let $n \geq 2$ and $m\geq 2$ be integers, let ${\V=\V_1\otimes \dots \otimes \V_m}$ be a vector space over a field $\field$, and let
\begin{align}
{E=\{x_a: a \in[n] \} \subseteq \pro{\V_1:\dots : \V_m}}
\end{align}
be a multiset of product tensors.
For each index $j \in [m]$, let $k_j=\krank(x_{a,j}: a \in [n])$ and ${d_j=\dim\spn\{x_{a,j} : a \in [n]\}}$. Define
\begin{align}\label{eq:q}
\mu=\max_{\substack{i, j \in [m]\\i \neq j}}\{d_i-k_i+d_j-k_j\}.
\end{align}
If for every index $i \in [m]$ it holds that
\begin{align}\label{balanced_k_rank}
k_i \leq \sum_{\substack{j \in [m]\\ j \neq i}} (k_j-1)+1,
\end{align}
then
\begin{align}\label{tensor_rank_inequality}
\rank\bigg[\sum_{a\in [n]} x_a\bigg] \geq \min\bigg\{n,\mu+\sum_{j=1}^m (k_j-1)+2-n\bigg\}.
\end{align}
\end{theorem}
Intuitively, the condition~\eqref{balanced_k_rank} ensures that the k-ranks are sufficiently balanced. This inequality is satisfied, for example, when the product tensors are symmetric. While we are unaware whether the precise inequality~\eqref{balanced_k_rank} is necessary for the lower bound~\eqref{tensor_rank_inequality} to hold, the following example illustrates that some inequality of this form must hold:

\begin{example}\label{ex:independent}
The set of product tensors
\begin{align}
E=\{e_1^{\otimes 3}, e_2^{\otimes 3}, e_3^{\otimes 3}, e_4^{\otimes 3}, e_5 \otimes (e_1+e_2)^{\otimes 2}, e_6\otimes (e_1-e_2)^{\otimes 2}\}
\end{align}
does not satisfy~\eqref{tensor_rank_inequality}. Indeed,
\begin{align}
\rank[\Sigma(E)]&=5\\
			&< q+k_1+k_2+k_3-1-n\\
			&=d_2+d_3-1\\
			&=7.
\end{align}
This example illustrates that in order for the bound~\eqref{tensor_rank_inequality} to hold, the k-ranks must be sufficiently ``balanced" in order to avoid cases such as this. In particular, some inequality resembling~\eqref{balanced_k_rank} is necessary. We remark that this example can be extended to further parameter regimes using Derksen's example~\cite{DERKSEN2013708}, and similar arguments as in Sections~\ref{sharp_tensor_bound} and~\ref{waring_non-rank_sharp}.
\end{example}

Note that when $m=2$, Theorem~\ref{tensor_cor} states that
\begin{align}
\rank\bigg[\sum_{a \in [n]} x_a\bigg] \geq d_1+d_2-n,
\end{align}
provided that $k_1=k_2$. This is Sylvester's matrix rank inequality (although Sylvester's result holds also when $k_1 \neq k_2$) \cite{horn2013matrix}.

%When $m=3$, Theorem~\ref{tensor_cor} implies that
%\begin{align}
%\rank\bigg[\sum_{a \in [n]} x_a\bigg] \geq \min\{n, d_1+d_2+k_3-1-n\},
%\end{align}
%whenever the k-ranks are balanced.
The following example demonstrates that our two lower bounds on tensor rank in Theorem~\ref{tensor_cor} and Corollary~\ref{tr_lower} are independent.
\begin{example}
By Theorem~\ref{tensor_cor}, the sum of the set of product tensors
\begin{align}
\{e_1^{\otimes 3},e_2^{\otimes 3}, (e_1+e_2)^{\otimes 2}\otimes e_3,e_3^{\otimes 2} \otimes (e_1+e_2+e_3)\}
\end{align}
has tensor rank $4$. Note that this bound cannot be achieved with the flattening rank lower bound, nor with Corollary~\ref{tr_lower}, as the first three vectors do not satisfy~\eqref{tr_lower_eq}. Many more such examples can be obtained using the construction in Section~\ref{sharp_tensor_bound}.

Conversely, the sum of the set of product tensors
\begin{align}\{& e_1^{\otimes 3}, e_2^{\otimes 3}, e_3^{\otimes 3}, e_4^{\otimes 3},(e_2+e_3)\otimes (e_2+e_4)\otimes (e_1+e_4)\}
\end{align}
has tensor rank 5 by Corollary~\ref{tr_lower}, while Theorem~\ref{tensor_cor} only certifies that this sum has tensor rank at least $4$.
\end{example}
Now we prove Theorem~\ref{tensor_cor}.
%Intuitively, the three conditions respectively state that the k-ranks, standard ranks, and differences between these two are sufficiently balanced. For example, these conditions are automatically satisfied for symmetric product tensors. As a result, Theorem~\ref{tensor_cor} has a succinct corollary for symmetric tensor decompositions, which we state in Corollary~\ref{symmetric_cor}. We note that Conditions~1 and~3 can be removed if one takes $q=\min_j \{d_j-k_j\}$, instead of $q=\max_j \{d_j-k_j\}$, and follows a very similar proof as below. Theorem~\ref{tensor_cor} also can be modified in the case when Conditions~2 and~3 are only simultaneously satisfied for some non-empty subset $J \subseteq [m]$, by instead defining $q=\max_{j \in J} \{d_j-k_j\}$. The advantage of our chosen incarnation of Theorem~\ref{tensor_cor} is that the inequality~\eqref{tensor_rank_inequality} is sharp, as we observe immediately after the proof.
\begin{proof}[Proof of Theorem~\ref{tensor_cor}]
Let $r =\rank[\sum_{a \in [n]} x_a]$, and let $\{y_a : a \in [r]\}\subseteq \pro{\V_1: \dots : \V_m}$ be a multiset of product tensors for which $\sum_{a \in [n]} x_a = \sum_{a \in [r]} y_a$ is a tensor rank decomposition. We need to prove that $r$ satisfies the inequality~\eqref{tensor_rank_inequality}. For each $a \in [r]$, let $x_{n+a}=-y_a$, and let ${T_1\sqcup \dots \sqcup T_t =[n+r]}$ be the index sets of the connected components of ${\{x_a : a \in [n+r]\}}$. For each subset ${S \subseteq [n]}$ and index $j \in [m]$, let
\begin{align}
d_j^S=\dim\spn\{x_{a,j} : a \in S\}.
\end{align}

We first consider the case $t=1$, i.e. $\{x_a : a \in [n+r]\}$ is connected. By the splitting theorem, it holds that
\begin{align}
n+r &\geq \sum_{j=1}^m (d_j-1)+2\\
	& \geq \mu+\sum_{j=1}^m (k_j-1)+2,
\end{align}
completing the proof in this case.

We proceed by induction on $t$. Suppose the theorem holds whenever the number of connected components is less than $t$. Assume without loss of generality that

\begin{align}
\bigabs{T_1 \cap [n]}-\bigabs{T_1 \cap [n+r]\setminus [n]} &\geq \bigabs{T_2 \cap [n]}-\bigabs{T_2 \cap [n+r]\setminus [n]}\\
 &\;\; \vdots \\
 & \geq \bigabs{T_{{t}} \cap [n]}-\bigabs{T_{{t}} \cap [n+r]\setminus [n]}\\
 &\geq 0,
\end{align}
where the last line follows from the fact that $\sum_{a \in [r]} y_a$ is a tensor rank decomposition. If
\begin{align}
\bigabs{T_1 \cap [n]} = \bigabs{T_1 \cap [n+r]\setminus [n]},
\end{align}
then $r=n$ and we are done. Otherwise,
\begin{align}
\bigabs{T_{[t-1]} \cap [n]} > \bigabs{T_{[t-1]} \cap [n+r]\setminus [n]},
\end{align}
where $T_{[t-1]}=T_1 \sqcup\dots \sqcup T_{t-1}$.

Observe that $k_j < \bigabs{T_{[t-1]}\cap[n]}$ for all $j \in [m]$. Indeed, since
\begin{align}
\rank\bigg[\sum_{a \in T_{[t-1]} \cap [n]} x_a\bigg] < \bigabs{T_{[t-1]} \cap [n]},
\end{align}
it must hold that
\begin{align}
2 \bigabs{T_{[t-1]} \cap [n]}-1 \geq \sum_{j=1}^m \bigg(\min\big\{\bigabs{T_{[t-1]}\cap [n]},k_j\big\}-1\bigg)+2,
\end{align}
by~\eqref{k-rank_bound1}. If $k_i \geq \bigabs{T_{{[t-1]}}\cap [n]}$ for some $i \in [m]$, then this inequality implies that ${k_j < \bigabs{T_{{[t-1]}}\cap [n]}}$ for all $j \neq i$, and hence
\begin{align}
k_i \geq \bigabs{T_{{[t-1]}}\cap [n]} \geq \sum_{\substack{j \in [m]\\ j \neq i}} (k_j-1)+2,
\end{align}
contradicting~\eqref{balanced_k_rank}. So $k_j < \bigabs{T_{{[t-1]}}\cap [n]}$ for all $j \in [m]$.

Since $k_j< \bigabs{T_{{[t-1]}}\cap [n]}$ for all $j \in [m]$, the k-ranks of $\{x_a : a \in T_{[t-1]}\cap [n]\}$ satisfy~\eqref{balanced_k_rank}, so by the induction hypothesis,
\begin{align}\label{induct_hat}
\bigabs{T_{[t-1]}} \geq \mu^{T_{[t-1]}\cap[n]}+\sum_{j=1}^m(k_j-1)+2,
\end{align}
where
\begin{align}
\mu^{T_{[t-1]}\cap[n]}=\max_{\substack{i,j \in [m]\\i\neq j}}\bigg\{d_i^{T_{[t-1]}\cap[n]}-k_i+d_j^{T_{[t-1]}\cap[n]}-k_j\bigg\}.
\end{align}
%Since $\{v_a^{\otimes m} : a \in T_t\}$ is connected, it holds that
%\begin{align}
%\abs{T_t} \geq \max\bigg\{2\bigabs{T_t \cap [n]},\sum_{j=1}^m(d_j^{T_t\cap [n]}-1)+2\bigg\}.
%\end{align}
To complete the proof, we will show that
\begin{align}
\bigabs{T_{[t-1]}}+\abs{T_t} \geq \mu+\sum_{j=1}^m (k_j-1)+2.
\end{align}
Let $i,i' \in [m]$ be such that $\mu=d_i-k_i+d_{i'}-k_{i'}$. Then
\begin{align}
\bigabs{T_{[t-1]}}+\abs{T_t}&\geq d_i^{T_{[t-1]}\cap[n]}-k_i+d_{i'}^{T_{[t-1]}\cap[n]}-k_{i'}+\sum_{j=1}^m(k_j-1)+\sum_{j=1}^m(d_j^{T_t\cap [n]}-1)+4\\
					&\geq d_i^{T_{[t-1]}\cap[n]}-k_i+d_{i'}^{T_{[t-1]}\cap[n]}-k_{i'}+\sum_{j=1}^m(k_j-1)+d_{i}^{T_t\cap [n]}+d_{i'}^{T_t\cap [n]}+2\\
					&\geq d_i-k_i+d_{i'}-k_{i'}+\sum_{j=1}^m(k_j-1)+2\\
					&= \mu+\sum_{j=1}^m(k_j-1)+2,
\end{align}
where the first line follows from~\eqref{induct_hat} and the fact that $\{x_a : a \in T_t\}$ is connected, the second is obvious, the third is easy to verify (in matroid-theoretic terms, this is submodularity of the rank function), and the fourth is by definition. This completes the proof.
%If $k_1 \geq \bigabs{T_{{[t-1]}}\cap [n]},$ then $q^{T_{[t-1]}\cap[n]}=0$, and
%\begin{align}
%\bigabs{T_{[t-1]}}+\abs{T_t}&\geq (\abs{T_{[t-1]}\cap [n]}-1)+ \sum_{j=2}^m(k_j-1)+\sum_{j=1}^m(d_j^{T_t\cap [n]}-1)+4\\
%					&\geq (\abs{T_{[t-1]}\cap [n]}-1)+\sum_{j=2}^m(k_j-1)+2d_i^{T_t\cap [n]}+2\\
%					&\geq 2(d_i-k_i)+\sum_{j=1}^m(k_j-1)+2\\
%					&\geq 2q+\sum_{j=1}^m(k_j-1)+2,
%\end{align}
%where the first line follows from~\eqref{induct_hat} and the fact that $\{x_a : a \in T_t\}$ is connected, the second follows from~\eqref{balanced}, the third uses submodularity of the rank function and the fact that $k_1 \geq \bigabs{T_{{[t-1]}}\cap [n]}$, and the fourth is obvious. This completes the proof.
\end{proof}

%
%For an integer $m \geq 2$ and a vector space $\W$ over a field $\field$, we say that a tensor $v \in \W^{\otimes m}$ is \textit{symmetric} if it is invariant under permutations of the subsystems. The \textit{Waring rank} of a symmetric tensor $v$ is the minimum number $n$ for which
%\begin{align}
%v = \sum_{a \in [n]} \alpha_a v_a^{\otimes m}
%\end{align}
%for some set of vectors $\{v_1,\dots, v_n\} \subseteq \W$ and scalars $\{\alpha_1,\dots, \alpha_n\} \subseteq \field$. We present two results on symmetric tensor decompositions in this section. First, we state a lower bound on Waring rank that follows immediately from Theorem~\ref{tensor_cor}, and observe that it is sharp. Second, we prove a result of a somewhat similar flavour to those presented in Section~\ref{non-rank}, which controls the number of symmetric tensors that can appear in a (potentially non-Waring-rank) decomposition of a symmetric tensor. As a corollary to the latter result, we prove a new identifiability criterion for uniqueness of Waring rank decompositions.

As an immediate corollary to Theorem~\ref{tensor_cor}, we obtain the following lower bound on the Waring rank of a symmetric tensor, in terms of a known symmetric decomposition.

\begin{cor}[Waring rank lower bound]\label{symmetric_cor}
Let $n\geq 2$, and $m \geq 2$ be integers, let $\W$ be a vector space over a field $\field$ with $\Char(\field)=0$ or $\Char(\field)>m$, and let ${\{v_a : a \in [n]\} \subseteq \W\setminus\{0\}}$ be a multiset of non-zero vectors. Let
\begin{align}k={\krank(v_a : a \in [n])}
\end{align}
 and
 \begin{align}
 {d=\dim\spn\{v_a : a \in [n]\}}.
 \end{align}
  Then for any multiset of non-zero scalars
\begin{align}
\{\alpha_a : a \in [n]\} \subseteq \field^\times,
\end{align}
it holds that 
\begin{align}\label{Waring_rank_inequality}
\setft{WaringRank}\bigg[ \sum_{a \in [n]} \alpha_a v_a^{\otimes m}\bigg] \geq \min\{n,2d+(m-2)(k-1)-n\}.
\end{align}
\end{cor}

\subsection{Our tensor rank lower bound is sharp}\label{sharp_tensor_bound}

In this subsection, we observe that, in a wide parameter regime, the inequalities~\eqref{tensor_rank_inequality} and~\eqref{Waring_rank_inequality} appearing in Theorem~\ref{tensor_cor} and Corollary~\ref{symmetric_cor} cannot be improved. 

Let $\field$ be a field with $\setft{Char}(\field)=0$, let $n \geq 2$, ${m \geq 2}$,
\begin{align}
2 \leq d_1,\dots, d_m \leq n,
\end{align}
and
\begin{align}
{k_1\leq d_1,\dots, k_m \leq d_m}
\end{align}
be positive integers, and let
\begin{align}
\lambda=\sum_{j=1}^m (k_j-1)+2.
\end{align}
Suppose that the following conditions hold:
\begin{enumerate}
\item ${\mu=2(d_{i}-k_i)}$ for some index $i \in [m]$, where $\mu$ is defined as in~\eqref{eq:q}.
\item$\max\{k_j: j \in [m]\}+d_i-k_i+1 \leq n \leq d_i-k_i+\lambda$
%\begin{align}
%\max\big\{\{d_1,\dots, d_m\} \sqcup \{d_i-k_i+k_j+1 : j \in [m]\}\big\} \leq n \leq \sum_{j=1}^m (k_j-1)+2,
%\end{align}
\item The inequality \eqref{balanced_k_rank} is satisfied.
\end{enumerate}
Then there exists a multiset of product tensors $E$ corresponding to these choices of parameters that satisfies~\eqref{tensor_rank_inequality} with equality. Indeed, the bound $\rank[\Sigma(E)] \geq n$ is trivial to attain with equality, and the bound
\begin{align}\label{eq:rank_bound}
\rank[\Sigma(E)] \geq 2(d_i-k_i)+\lambda-n
\end{align}
can be attained with equality as follows. Let
\begin{align}
\{x_a : a \in [\lambda]\} \subseteq \pro{\field^{d_1}: \dots : \field^{d_m}}
\end{align}
be a multiset of product tensors that forms a circuit and satisfies
\begin{align}\label{eq:l}
\dim\spn\{x_{a,j}: a \in [\lambda]\}=\krank(x_{a,j} : a \in [\lambda])=k_j
\end{align}
for all $j \in [m]$. An example of such a circuit is presented in~\cite{DERKSEN2013708}, and reviewed in Section~\ref{conjecture_kruskal}. Now, let
\begin{align}
\{x_{a} : a\in [\lambda+d_i-k_i]\setminus [\lambda]\}\subseteq \pro{\field^{d_1}: \dots : \field^{d_m}}
\end{align}
be any multiset of product tensors for which
\begin{align}\label{eq:other}
\dim\spn\{x_{a,j} : a \in [\lambda+d_i-k_i]\}=d_j
\end{align}
and
\begin{align}
\krank(x_{a,j} : a\in [\lambda+d_i-k_i]) = k_j
\end{align}
for all $j \in [m]$, which is guaranteed to exist since $\field$ is infinite. Let
\begin{align}
E=\{x_a : a \in [n-d_i+k_i]\} \sqcup \{x_a : a \in [\lambda+d_i-k_i]\setminus [\lambda]\}
\end{align}
and
\begin{align}
F=\{x_a : a \in [\lambda] \setminus [n-d_i+k_i]\} \sqcup \{x_a : a \in [\lambda+d_i-k_i]\setminus [\lambda]\}.
\end{align}

Recall that $n \leq d_i-k_i+\lambda$ by assumption, so the set $[\lambda] \setminus [n-d_i+k_i]$ that appears in the definition of $F$ is well-defined. Since $n -d_i+k_i \geq k_j+1$ for all $j \in [m]$, $E$ has k-ranks $k_1,\dots, k_m,$ as desired. It is also clear that $E$ has ranks $d_1,\dots,d_m$, by~\eqref{eq:l} and~\eqref{eq:other}.
%\begin{align}
%E=\{x_a : a \in [l+d_i-k_i]\setminus[l-n]\} \sqcup \{x_a : a \in [l+d_i-k_i]\setminus [l]\}
%\end{align}
%and
%\begin{align}
%F=\{x_a : a \in [l] \setminus [n-d_i+k_i]\} \sqcup \{x_a : a \in [l+d_i-k_i]\setminus [l]\}.
%\end{align}
Since $\{x_a: a \in [\lambda]\}$ forms a circuit, some non-zero linear combination of $E$ is equal to a non-zero linear combination of $F$. Since $\abs{F}$ is equal to the right hand side of~\eqref{eq:rank_bound}, this completes the proof.
% linear combination of $E$ must satisfy~\eqref{tensor_rank_inequality} with equality, completing the construction.

Out of the three conditions required for our construction, $\mu=2(d_i-k_i)$ seems the most restrictive. Unfortunately, our methods appear to require this condition. A nearly identical construction shows that the inequality~\eqref{Waring_rank_inequality} appearing in Corollary~\ref{symmetric_cor} cannot be improved (and our restrictive condition on $\mu$ is automatically satisfied in this case). The only difference in the construction is to choose the product tensors $\{x_a : a \in [\lambda+d_i-k_i]\}$ to be symmetric in this case, which can always be done (in particular, the product tensors appearing in Derksen's example can be taken to be symmetric).

%-----------------------------------------------------------------------
\section{A uniqueness result for non-Waring rank decompositions}\label{symmetric_non-rank}
%-----------------------------------------------------------------------

In this section, we prove a sufficient condition on a symmetric decomposition
\begin{align}\label{eq:symmetric_non-rank}
v=\sum_{a \in [n]} \alpha_a v_a^{\otimes m}
\end{align}
under which any distinct decomposition $v=\sum_{a \in [r]} \beta_a u_a^{\otimes m}$ must have $r$ lower bounded by some quantity, which we call $r_{\min}$ for now. When $r_{\min} \leq n$, this yields a lower bound on $\setft{WaringRank}(v)$ that is contained in Corollary~\ref{symmetric_cor}. When ${r_{\min} =n+1}$, this yields a uniqueness criterion for symmetric decompositions that is contained in Theorem~\ref{k-gen}, but improves Kruskal's theorem in a wide parameter regime. The main result in this section is the case ${r_{\min} > n+1}$, where we obtain an even stronger statement than uniqueness: Every symmetric decomposition of $v$ into less than $r_{\min}$ terms must be equal to $\sum_{a \in [n]} \alpha_a v_a^{\otimes m}$ (in the language introduced in Section~\ref{mp}, $\sum_{a \in [n]} \alpha_a v_a^{\otimes m}$ is the \textit{unique symmetric decomposition of $v$ into less than $r_{\min}$ terms}). In Section~\ref{waring_non-rank_sharp} we prove that our bound $r_{\min}$ cannot be improved. In Section~\ref{non-rank_applications} we identify potential applications of our non-rank uniqueness results.

Our results in this section were inspired by, and generalize, Theorem~6.8 and Remark~6.14 in~\cite{Chiantini2019}. Our results in this section should be compared with those of Section~\ref{non-rank} on uniqueness of non-rank decompositions of tensors that are not necessarily symmetric.

\begin{theorem}\label{symmetric_non-rank_theorem}
Let $n\geq 2$ and $m \geq 2$ be integers, let $\W$ be a vector space over a field $\field$ with $\Char(\field)=0$ or $\Char(\field)>m$, let $E={\{v_a : a \in [n]\} \subseteq \W\setminus\{0\}}$ be a multiset of non-zero vectors with ${\krank(v_a : a \in [n]) \geq 2}$, and let
\begin{align}
{d=\dim\spn\{v_a : a \in [n]\}}.
\end{align}
Then for any non-negative integer $r \geq 0$, multiset of non-zero vectors ${F={\{u_a : a \in [r]\} \subseteq \W\setminus\{0\}}}$ with ${\krank(u_a : a \in [r]) \geq \min\{2,r\}}$,
and multisets of non-zero scalars
\begin{align}
\{\alpha_a : a \in [n]\}, \{\beta_a: a \in [r]\} \subseteq \field^\times
\end{align}
for which 
\begin{align}\label{symmetric_notequal}
\{\alpha_a v_a^{\otimes m} : a \in [n]\} \neq \{\beta_a u_a^{\otimes m} : a \in [r]\}
\end{align}
and
\begin{align}\label{symmetric_equation}
\sum_{a \in [n]} \alpha_a v_a^{\otimes m} = \sum_{a \in [r]} \beta_a u_a^{\otimes m},
\end{align}
it holds that 
\begin{align}\label{symmetric_first_inequality}
n+ r \geq m+2d-2.
\end{align}
%If it furthermore holds that
%\begin{align}\label{disjointness}
%\{\alpha_a v_a^{\otimes m} : a \in [n]\} \cap \{\beta_a u_a^{\otimes m} : a \in [r]\}=\{\},
%\end{align}
%then
%\begin{align}\label{symmetric_second_inequality}
%n+ r \geq \min\{2(m+1),m(d-1)+1\} +1.
%\end{align}
\end{theorem}

In the language of the introduction to this section, $r_{\min} = m+2d-2-n$. For comparison, the result we have referred to in~\cite{Chiantini2019} asserts that, under the condition $n \leq m$, it holds that $n+r \geq m+d$, which is weaker than our bound~\eqref{symmetric_first_inequality}.

%In Section~\ref{waring_non-rank_sharp} we prove that our bound~\eqref{symmetric_first_inequality} is sharp. In fact, we prove that~\eqref{symmetric_first_inequality} is very nearly sharp even if the k-rank condition is tightened to the requirement that ${\krank(v_a : a \in [n]) \geq k}$ for some $k \geq 3$, in a certain parameter regime.

\begin{proof}[Proof of Theorem~\ref{symmetric_non-rank_theorem}]
By subtracting terms from both sides of~\eqref{symmetric_equation}, and combining parallel product tensors into single terms (or to zero), it is clear that it suffices to prove the statement when $E$ is linearly independent (so $d=n$).

Note that $r \geq n$ by Kruskal's theorem. For each $a \in [r]$, let $v_{n+a}=u_a$, and let ${T_1\sqcup \dots \sqcup T_t =[n+r]}$ be the index sets of the connected components of ${\{v_a^{\otimes m} : a \in [n+r]\}}$. Assume without loss of generality that $\bigabs{T_1 \cap [n]} \geq \dots \geq \bigabs{T_t \cap [n]}$, and let $\tilde{t}\in[t]$ be the largest integer for which $\bigabs{T_{\tilde{t}} \cap [n]} \geq 1$. By~\eqref{symmetric_notequal}, there must exist $\tilde{p} \in [\tilde{t}]$ for which $\abs{T_{\tilde{p}}} \geq 3$. Note that
\begin{align}
\dim\spn\{v_a : a \in T_{\tilde{p}}\} \geq \max\left\{2, \bigabs{T_{\tilde{p}} \cap [n]}\right\}.
\end{align}
Since $\{v_a^{\otimes m} : a \in T_{\tilde{p}}\}$ is connected, it follows from our splitting theorem that
\begin{align}\label{new_symmetric_proof_inequality}
\abs{T_{\tilde{p}}} \geq m(\max\left\{2, \bigabs{T_{\tilde{p}} \cap [n]}\right\}-1)+2.
\end{align}
Now,
\begin{align}
n+r &\geq \sum_{p \in [\tilde{t}]} \abs{T_p}\\
	%&\geq 2(\tilde{t}-\bigabs{T_p\cap [n]})+m(\max\{2, \bigabs{T_p \cap [n]}\}-1)+2\\
	& \geq \sum_{p \neq \tilde{p}} \big[m\left(\bigabs{T_p\cap [n]} -1\right)+2\big]+ m\left(\max\left\{2, \bigabs{T_{\tilde{p}} \cap [n]}\right\}-1\right)+2\\
	& = m\left(n-\abs{T_{\tilde{p}}\cap [n]}\right)-(m-2)\left(\tilde{t}-1\right)+m\left(\max\left\{2, \bigabs{T_{\tilde{p}} \cap [n]}\right\}-1\right)+2\\
	& \geq m\left(n-\abs{T_{\tilde{p}}\cap [n]}\right)-(m-2)\left(n- \bigabs{T_{\tilde{p}} \cap [n]}\right)+m\left(\max\left\{2, \bigabs{T_{\tilde{p}} \cap [n]}\right\}-1\right)+2\\
	&=2n-2\bigabs{T_{\tilde{p}} \cap [n]}+m\left(\max\left\{2, \bigabs{T_{\tilde{p}} \cap [n]}\right\}-1\right)+2\\
	%&=2n+\max\left\{m-2\left(\bigabs{T_{\tilde{p}} \cap [n]}-1\right), \bigabs{T_{\tilde{p}} \cap [n]}(m-2)+2\right\}\\
	&\geq 2n+m-2.
%	&\geq m(n-\bigabs{T_{\tilde{p}}\cap [n]})-(n-\bigabs{T_{\tilde{p}}\cap [n]}-1)(m-2)+m(\bigabs{T_p \cap [n]}-1)+2\\
%	&=2n+\bigabs{T_{\tilde{p}}\cap [n]}(m-2)\\
%	&\geq 2n+m-2,
\end{align}
The first line is obvious, the second follows from~\eqref{new_symmetric_proof_inequality} and the fact that every multiset $\{v_a^{\otimes m} : a \in T_p\}$ is connected, the third is algebra, the fourth uses the fact that $\bigabs{T_p \cap [n]} \geq 1$ for all $p \in [\tilde{t}]$, and the rest is algebra. This completes the proof.
\end{proof}
Theorem~\ref{symmetric_non-rank_theorem} immediately implies the following uniqueness result for non-Waring rank decompositions.

\begin{cor}[Uniqueness result for non-Waring rank decompositions]\label{waring_uniqueness}
Let $n\geq 2$ and $m \geq 2$ be integers, let $\W$ be a vector space over a field $\field$ with $\Char(\field)=0$ or $\setft{Char}(\field)>m$, let ${\{v_a : a \in [n]\} \subseteq \W\setminus\{0\}}$ be a multiset of non-zero vectors with ${\krank(v_a : a \in [n]) \geq 2}$, let $\{\alpha_a : a \in [n]\} \subseteq \field^\times$ be a multiset of non-zero scalars, and let ${d=\dim\spn\{v_a : a \in [n]\}}$. If
\begin{align}\label{waring_uniqueness_inequality}
2n+1 \leq m+2d-2,
\end{align}
then $\sum_{a \in [n]} \alpha_a v_a^{\otimes m}$ constitutes a unique Waring rank decomposition. More generally, if
\begin{align}\label{waring_uniqueness_inequality}
n +r +1\leq m+2d-2,
\end{align}
for some $r \geq n$, then $\sum_{a \in [n]} \alpha_a v_a^{\otimes m}$ is the unique symmetric decomposition of this tensor into at most $r$ terms.
\end{cor}

Note that the $r=n$ case of Corollary~\ref{waring_uniqueness} improves Kruskal's theorem for symmetric decompositions as soon as $2d > m(k-2)+4$, where $k=\krank(v_a : a \in [n])$. This case of Corollary~\ref{waring_uniqueness} is in fact contained in our generalization of Kruskal's theorem (Theorem~\ref{k-gen}), since for every subset $S\subseteq [n]$ with $2 \leq \abs{S} \leq n$, it holds that
\begin{align}
2 \abs{S}&=2n-2\bigabs{[n]\setminus S}\\
		&\leq m+2d-2\bigabs{[n]\setminus S}-3\\
		&\leq m+2d^{S}-3\\
		&\leq m(d^S-1)+1,
\end{align}
where $d^S=\dim\spn\{v_a : a \in S\}$. This demonstrates that our generalization of Kruskal's theorem is stronger than Kruskal's theorem, even for symmetric tensor decompositions.

Our main result in this section is the $r>n$ case of Corollary~\ref{waring_uniqueness}, which yields uniqueness results for non-Waring rank decompositions of $\sum_{a \in [n]} \alpha_a v_a^{\otimes n}$. The following example illustrates this case in practice.
\begin{example}\label{ex:symmetric}
It follows from Corollary~\ref{waring_uniqueness} that for any positive integers $m\geq 3$ and $n \geq 2$, $\sum_{a \in [n]} e_a^{\otimes m}$ is the unique symmetric decomposition of this tensor into at most $m+n-3$ terms.
\end{example}

It is natural to ask if Corollary~\ref{waring_uniqueness} can be improved under further restrictions on $\krank(v_a : a \in [n])$. At the end of Section~\ref{waring_non-rank_sharp} we prove that this cannot be done, at least in a particular parameter regime.

%-----------------------------------------------------------------------
\subsection{The inequality appearing in our uniqueness result is sharp}\label{waring_non-rank_sharp}
%-----------------------------------------------------------------------

In this subsection we prove that the inequality~\eqref{symmetric_first_inequality} that appears in Theorem~\ref{symmetric_non-rank_theorem} cannot be improved, by constructing explicit multisets of symmetric product tensors that satisfy this bound with equality.

Let $\field$ be a field with $\Char(\field)=0$. We will prove that for any choice of positive integers $m \geq 2$, $d \geq 2$, $r \geq d-2$, and $n\geq d$ for which $n+r=m+2d-2$, there exist multisets of non-zero vectors $E$ and $F$ that satisfy the assumptions of Theorem~\ref{symmetric_non-rank_theorem}. Note that the inequality ${r \geq d-2}$ automatically holds when $r \geq n$, so this assumption does not restrict the parameter regime in which the inequality appearing in our uniqueness result (Corollary~\ref{waring_uniqueness}) is sharp as a consequence.

We first consider the case $d=2$. Let $\{v_a^{\otimes m} : a \in [m+2]\} \subseteq \pro{\field^{2} : \dots : \field^2}$ be a circuit of symmetric product tensors for which
\begin{align}
\krank(v_a: a \in [m+2])= 2.
\end{align}
An example of such a circuit is given in~\cite{DERKSEN2013708}, and reviewed in Section~\ref{conjecture_kruskal}. So there exist non-zero scalars $\{\alpha_a : a \in [m+2]\}\subseteq \field^\times$ for which $\sum_{a \in [m+2]} \alpha_a v_a^{\otimes m}=0,$ and we can take the multisets $E=\{v_a : a \in [n]\}$ and $F=\{v_a : a \in [m+2]\setminus [n]\}$ to conclude.

For $d \geq 3$, let $\{v_a^{\otimes m} : a \in [m+2]\} \subseteq \pro{\field^{d} : \dots : \field^d}$ be the same multiset of symmetric product tensors as above, embedded in a larger space. Let
\begin{align}
\{v_a: a \in [d+m]\setminus [m+2]\} \subseteq \field^d\setminus \{0\}
\end{align}
be any multiset of non-zero vectors for which
\begin{align}
\dim\spn\{v_a : a \in [d+m]\}=d
\end{align}
and
\begin{align}
\krank\{v_a : a \in [d+m]\}\geq 2,
\end{align}
which is guaranteed to exist since $\field$ is infinite. Since $r \geq d-2$, we can take the multisets
\begin{align}
E=\{v_a: a \in [n-d+2]\}\sqcup\{v_a: a \in [d+m]\setminus [m+2]\}
\end{align}
and
\begin{align}
F=\{v_a : a \in [m+2]\setminus[n-d+2]\} \sqcup \{v_a: a \in [d+m]\setminus [m+2]\}
\end{align}
to conclude.

Somewhat surprisingly, the inequality~\eqref{symmetric_first_inequality} is very nearly sharp even when the k-rank condition is tightened to ${\krank(v_a : a \in [n]) \geq k}$ for some $k \geq 3$, under certain parameter constraints. More specifically, for any $k \in \{3,4,\dots, d-1\}$, it is almost sharp under the choice $n=d+1$ and $r=m+d-1$. Let
\begin{align}
E= \{v_a : a \in [d+m]\setminus[m]\} \sqcup \bigg\{\sum_{a\in [k]} v_a\bigg\},
\end{align}
and
\begin{align}
F=\{v_a: a \in [m]\} \sqcup \{v_a : a \in [d+m]\setminus [m+2] \}\sqcup \bigg\{\sum_{a\in [k]} v_a\bigg\}.
\end{align}
Here, $\abs{E}+\abs{F}=2d+m,$ exceeding our lower bound by 2. When $k=d$, take the same multisets $E$ and $F$, with $\sum_{a\in [d]} v_a$ removed, to observe that our bound is sharp under the choice $n=d$ and $r=m+d-2$. Note that the k-rank is brought down to $k$ because of a single vector in the multiset. This is a concrete demonstration of the fact that the k-rank is a very crude measure of genericity. We emphasize that this construction relies on the particular choice of parameters $n=d+1$, and $r=m+d-1$. It is possible that the inequality~\eqref{symmetric_first_inequality} could be significantly stengthened for other choices of $n$ and $r$. Indeed, we have exhibited such an improvement for $r \leq n$ in Corollary~\ref{symmetric_cor}.

%We have proven that Theorem~\ref{symmetric_cor} cannot be substantially improved with different k-rank conditions. However, it may still be possible to refine this result in a similar manner to our generalization of Kruskal's theorem, namely, by replacing the k-rank condition with more precise rank conditions. We do not explore this possibility in the present work.

%-----------------------------------------------------------------------
\subsection{Applications of non-rank uniqueness results}\label{non-rank_applications}
%-----------------------------------------------------------------------

In this subsection, we identify potential applications of our results on uniqueness of non-rank decompositions. For concreteness, we focus on the symmetric case and our non-Waring rank uniqueness result in Corollary~\ref{waring_uniqueness}, however similar comments can be applied to our analogous results in Section~\ref{non-rank} in the non-symmetric case.

We say a symmetric tensor $v$ is \textit{identifiable} if it has a unique Waring rank decomposition. For the purposes of this discussion, we will say that $v$ is $r$\textit{-identifiable} for some ${r \geq \rank(v)}$ if the Waring rank decomposition of $v$ is the unique symmetric decomposition of $v$ into at most $r$ terms (see Section~\ref{mp}). Corollary~\ref{waring_uniqueness} provides a sufficient condition for a symmetric tensor $v$ to be $r$-identifiable for $r>\rank(v)$, and Example~\ref{ex:symmetric} demonstrates the existence of symmetric tensors satisfying this condition. We can thus define a hierarchy of identifiable symmetric tensors (of some fixed rank), where those that are $r$-identifiable for larger $r$ can be thought of as ``more identifiable." We suggest that studying this hierarchy could be a useful tool for studying symmetric tensor decompositions. For example, although most symmetric tensors of sub-generic rank are identifiable, it is notoriously difficult to find the rank decomposition of such tensors~\cite{landsberg2012tensors,pmlr-v35-bhaskara14b,chiantini2017generic}. Perhaps one can leverage the additional structure of $r$-identifiable symmetric tensors to find efficient decompositions.

In applications, one often has a symmetric decomposition of a tensor, and wants to control the possible symmetric decompositions with fewer terms. Uniqueness results for non-rank decompositions can be turned around to apply in this setting: Suppose we know that if a symmetric decomposition into $n$ terms satisfies some condition, call it $C$, then it is the unique symmetric decomposition into at most $r$ terms, for some $r > n$. Then if one starts with a symmetric decomposition of a symmetric tensor $v$ into $r$ terms, she knows that there are no symmetric decompositions of $v$ into $n<r$ terms that satisfies condition $C$. In this way, one can use a non-rank uniqueness result to control the possible decompositions of $v$ into fewer than $r$ symmetric product tensors. Applying this reasoning to our Corollary~\ref{waring_uniqueness} simply yields a special case of Theorem~\ref{symmetric_non-rank_theorem}. However, applying analogous reasoning to Corollary~\ref{k_arbitrarys1} in the non-symmetric case seems to produce new results.

%-----------------------------------------------------------------------
\section{Comparing our generalization of Kruskal's theorem to the uniqueness criteria of Domanov, De Lathauwer, and S\o{}rensen}\label{uniqueness_applications}
%-----------------------------------------------------------------------

In this section we compare our generalization of Kruskal's theorem to uniqueness criteria obtained by Domanov, De Lathauwer, and S\o{}rensen (DLS) in the case of three subsystems \cite{domanov2013uniqueness,domanov2013uniqueness2,domanov2014canonical,sorensen2015new,sorensen2015coupled}, which are the only previously known extensions of Kruskal's theorem that we are aware of. A drawback to the uniqueness criteria of DLS is that, similarly to Kruskal's theorem, they require the k-ranks to be above a certain threshold. In Section~\ref{threshold} we make this statement precise, and show by example that our generalization of Kruskal's theorem can certify uniqueness below this threshold. Moreover, in Section~\ref{combine} we observe that our generalization of Kruskal's theorem contains many of the uniqueness criteria of DLS. The uniqueness criteria of DLS are spread across five papers, and can be difficult to keep track of. For clarity and future reference, in Theorem~\ref{uniqueness} we combine all of these criteria into a single statement. In~Section~\ref{new_conjecture} we use insight gained from this synthesization and our Theorem~\ref{k-gen} as evidence to support a conjectural uniqueness criterion that would contain and unify every uniqueness criteria of DLS into a single, elegant statement.

For the remainder of this section, we fix a vector space $\V=\V_1 \otimes \V_2 \otimes \V_3$ over a field $\field$, and a multiset of product tensors
\begin{align}
\{x_a : a \in [n]\} \subseteq \pro{\V_1:\V_2 : \V_3}
\end{align}
with k-ranks $k_j=\krank(x_{a,j} : a \in [n])$ for each $j \in [3]$. For each subset $S \subseteq [n]$ with $2 \leq \abs{S} \leq n$ and index $j \in [3]$, we let
\begin{align}
d_j^S=\dim\spn\{x_{a,j} : a \in [n]\}.
\end{align}
We also let $d_j=d_j^{[n]}$ for all $j \in [3]$.

%----------------------------------------------------------------------------------%
\subsection{Uniqueness below the k-rank threshold of DLS}\label{threshold}
%----------------------------------------------------------------------------------%

All of the uniqueness criteria of DLS require the k-ranks to be above a certain threshold. In this subsection, we show by example that our generalization of Kruskal's theorem can certify uniqueness below this threshold.

Making this threshold precise, the uniqueness criteria of DLS cannot be applied whenever
\begin{align}\label{k-rank-large}
&\min\{k_2,k_3\} \leq n-d_1+1,\nonumber\\
\text{ and } &\min\{k_1,k_3\} \leq n-d_2 + 1,\nonumber\\
\text{ and } & \min\{k_1,k_2\} \leq n-d_3 +1.
%\begin{cases}
%\min\{k_2,k_3\} \geq n-d_1+2\\
%\min\{k_1,k_3\} \geq n-d_2 + 2\\
%\min\{k_1,k_2\} \geq n-d_3 +2
%\end{cases}
\end{align}
For example, if $k_2= k_3=2$, then the uniqueness criteria of DLS can only certify uniqueness if $d_1=n$. The following example shows that our generalization of Kruskal's theorem (Theorem~\ref{k-gen}) can certify uniqueness even if~\eqref{k-rank-large} holds.

\begin{example}\label{certify}
%[See example 5.2 in \cite{domanov2013uniqueness2}]
Consider the multiset of product tensors
\label{uniqueness_example}
\begin{align}\{&\alpha_1 e_1^{\otimes 3},\alpha_2 e_2^{\otimes 3}, \alpha_3 e_3^{\otimes 3},\alpha_4 e_4^{\otimes 3},\alpha_5(e_2+e_3)\otimes (e_2+e_4)\otimes (e_1+e_4)\} \quad \text{for} \quad \alpha_1,\dots, \alpha_5 \in \field^{\times}.
\end{align}
In this example, $k_1=k_2=k_3=2$, $d_1=d_2=d_3=4$, and $n-d_j+2=3$ for all $j \in [3]$, so~\eqref{k-rank-large} holds. Nevertheless, for arbitrary $\alpha_1,\dots, \alpha_5 \in \field^{\times}$, our generalization of Kruskal's theorem certifies that the sum of these product tensors constitutes a unique tensor rank decomposition. We note that uniqueness for $\alpha_2=\dots=\alpha_5=1$ was proven in~\cite[Example~5.2]{domanov2013uniqueness2}, using a proof specific to this case, in order to demonstrate that their uniqueness criteria are not also necessary for uniqueness.
%%In fact, an even less restrictive condition, Condition~W below, does not hold, so even the difficult-to-check uniqueness conditions proposed in~\cite{domanov2013uniqueness2} involving Condition~W do not apply. 
%In~\cite{domanov2013uniqueness2}, it is proven that if $\alpha_2=\dots=\alpha_5=1$, then the sum of these product tensors constitutes a unique tensor rank decomposition.
%%(This is done in order to show that Condition~W is not necessary for uniqueness.)
%The proof given in~\cite{domanov2013uniqueness2} is quite complicated and specific to this case. This is to be expected, as uniqueness does not follow directly from their uniqueness criteria. Uniqueness for arbitrary $\alpha_1,\dots, \alpha_5 \in \field^{\times}$ follows easily from Theorem~\ref{k-gen}.
\end{example}

Example~\ref{certify} shows that Theorem~\ref{k-gen} is strictly stronger than Kruskal's theorem, and is independent of the uniqueness criteria of DLS. It is natural to ask if Theorem~\ref{k-gen} is stronger than Kruskal's theorem even for symmetric tensor decompositions. We have observed in Section~\ref{symmetric_non-rank} that this is indeed the case.

%-------------------------------------------------------------------------------------------%
\subsection{Extending several uniqueness criteria of DLS}\label{combine}
%-------------------------------------------------------------------------------------------%

In this subsection, we observe that several of the uniqueness criteria of DLS are contained in our generalization of Kruskal's theorem, and prove a further, independent uniqueness criterion. The uniqueness criteria of DLS are numerous, and can be difficult to keep track of. To more easily analyze these criteria, in Theorem~\ref{uniqueness} we combine them all into a single statement.

\subsubsection{Conditions~U,~H,~C, and ~S}
Here we introduce several different conditions on multisets of product tensors, which will make the uniqueness criteria of DLS easier to state, and also make them easier to relate to our generalization of Kruskal's theorem. We first recall Conditions~U,~H, and~C from \cite{domanov2013uniqueness,domanov2013uniqueness2}. For notational convenience, we have changed these definitions slightly from~\cite{domanov2013uniqueness,domanov2013uniqueness2}. For example, our Condition~U is their Condition~$U_{n-d_1+2}$, with the added condition that $k_1 \geq 2$. After reviewing Conditions~U,~H, and~C, we introduce Condition~S, which captures the conditions of our generalization of Kruskal's theorem in the case ${m=3}$. Unlike Conditions~U,~H, and~C, our Condition~S does not appear in~\cite{domanov2013uniqueness,domanov2013uniqueness2}, nor anywhere else that we are aware of.

For a vector $\alpha \in \field^n$, we let $\omega(\alpha)$ denote the number of non-zero entries in $\alpha$.

%
%Note that the question of whether $\sum_{a \in [n]} x_a$ constitutes a unique tensor rank decomposition is theoretically computable, as it can be formulated as the following ideal membership problem: Is the (Zariski closed) set
%\begin{align}
%\{(y_1,\dots, y_n)\in \pro{\V_1:\V_2 : \V_3}^{\times n} : \sum_a y_a=\sum_a x_a\} \subseteq \V^{\times n}
%\end{align}
%contained in the Zariski closed set
%\begin{align}
%\bigcup_{\sigma \in S_n}\{(y_1,\dots, y_n)\in \pro{\V_1:\V_2 : \V_3}^{\times n} : y_{a}=x_{\sigma(a)} \quad \text{for all} \quad a \in [n]\} \subseteq \V^{\times n}?
%\end{align}
%Such questions can theoretically be answered deterministically using Gröbner bases, however these algorithms often require exponential space in the number of variables (even to certify membership of a single polynomial in an ideal). The number of variables in this problem is $n(d_1+d_2+d_3)$, so one would strongly expect this problem to quickly become intractable for large $n$.

%\begin{condition}[See Equation~1.12 in \cite{domanov2013uniqueness2}]\label{1.12}
%There exists $\sigma \in S_3$ such that
%\begin{align}\label{1.12eq}
%k_{\sigma(1)}+\max\{ \min\{k_{\sigma(2)},k_{\sigma(3)}-1\},\min\{k_{\sigma(2)}-1,k_{\sigma(3)}\}\}\geq n+1.
%\end{align}
%\end{condition}

%\begin{namedtheorem}{Condition W}
%It holds that $k_{1}\geq 2$, and for all $f \in \V_{1}^*$,
%\begin{align}
%\setft{rank}\Big[\sum_{a \in [n]} f(x_{a,1}) x_{a,2} \otimes x_{a,3}\Big]\geq n-d_1+2 \quad \text{whenever}\quad \omega(f(x_{1,1}),\dots,f(x_{n,1})) \geq n-d_{1}+2.
%\end{align}
%\end{namedtheorem}

\begin{namedtheorem}{Condition U}
It holds that $k_{1} \geq 2$, and for all $\alpha \in \field^n$,
\begin{align}\label{Umeq}
\rank\Big[\sum_{a \in [n]} \alpha_a x_{a, 2}\otimes x_{a,3} \Big]\geq \min\{\omega(\alpha),n-d_1+2\}.
\end{align}
\end{namedtheorem}

\begin{namedtheorem}{Condition H}
It holds that $k_{1} \geq 2$, and
\begin{align}\label{k-gen2eq}
d_{2}^S\!+d_{3}^S\!-\abs{S} \geq \min\{\abs{S}, n-d_{1}\!+2\}
\end{align}
for all $S \subseteq [n]$ with $2 \leq \abs{S}\leq n$.
\end{namedtheorem}
%Note that Condition~\ref{Wm} is less restrictive than Condition~\ref{Um}. Proposition 1.22 in~\cite{domanov2013uniqueness2} states that Conditions~\ref{1.12} and~\ref{Wm} imply uniqueness. Corollary 1.23 in~\cite{domanov2013uniqueness2} states that Conditions~\ref{1.12} and~\ref{Um} imply uniqueness. Proposition~1.26 in~\cite{domanov2013uniqueness2} states that if Condition~\ref{Um} holds for two distinct permutations $\sigma, \mu \in S_3$, then the decomposition is unique. These results are some of the most general known sufficient conditions for uniqueness, but unfortunately Conditions~\ref{Wm} and~\ref{Um} can be difficult to check. As a result, there are two more restrictive but easily checkable conditions that are often used instead. The first is easy to describe.
Condition~C takes a bit more work to describe. We use coordinates for this condition, in order to avoid having to introduce further multilinear algebra notation. For positive integers $q,r,$ and $t$, and matrices
\begin{align}
Y=(y_1,\dots, y_t) \in \Lin(\field^t, \field^q)\\
Z=(z_1,\dots, z_t)\in \Lin(\field^t, \field^r),
\end{align}
let
\begin{align}
Y \odot Z=(y_1\otimes z_1, \dots, y_t \otimes z_t) \in \Lin(\field^t,\field^{qr})
\end{align}
denote the \textit{Khatri-Rao product} of $Y$ and $Z$. Suppose $\V_j=\field^{d_j}$ for each $j \in [3]$, and consider the matrices
\begin{align}
X_j=(x_{1,j},\dots, x_{n,j})\in \Lin(\field^n,\field^{d_j})
\end{align}
for $j \in [3]$. For a positive integer $s\leq d_j$, let $\C_s(X_j)$ be the $\binom{d_j}{s} \times \binom{n}{s}$ matrix of $s \times s$ minors of $X_j$, with rows and columns arranged according to the lexicographic order on the size $s$ subsets of $[d_j]$ and $[n]$, respectively. Define the matrix
\begin{align}
C_s=\C_s(X_{2})\odot \C_s(X_{3}) \in \Lin(\field^{\binom{n}{s}}, \field^{q}),
\end{align}
where $q=\big(\substack{d_{2}\\s}\big)\big(\substack{d_{3}\\s}\big)$. Now we can state Condition~C.
%It turns out that Condition~W is equivalent to the existence of a permutation $\sigma \in S_3$ such that 
%\begin{align}
%\C_{s}^\sigma v \neq 0
%\end{align}
%for $s=n-d_{\sigma(1)}+2$ and a certain collection of non-zero vectors $v \in \field^{\binom{n}{l}}$. Condition~U is equivalent to the same statement, but for a larger collection of vectors. Therefore, Conditions~\ref{Wm} and~\ref{Um} are both implied by the following easily checkable condition.
\begin{namedtheorem}{Condition C}
It holds that $k_{1} \geq 2$, $\min\{d_2,d_3\} \geq n-d_1+2$, and
\begin{align}
\rank(C_{n-d_{1}+2})=\left(\substack{ n \\ n-d_{1}+2} \right).
\end{align}
\end{namedtheorem}
To more easily compare our generalization of Kruskal's theorem to the uniqueness criteria of DLS, we give a name (Condition~S) to the condition of our Theorem~\ref{k-gen} in the case $m=3$.
\begin{namedtheorem}{Condition S}
It holds that
\begin{align}\label{concon}
2 \abs{S} \leq d_1^S+d_2^S+d_3^S-2
\end{align}
for all $S \subseteq [n]$ with $2 \leq \abs{S}\leq n$.
\end{namedtheorem}

These conditions are related to each other as follows:
%\begin{equation}
%\begin{tikzcd}[row sep=tiny]
%\text{Condition~H} \arrow[Rightarrow,dr] & & &\\
% & \hspace{-5pt} & \text{Condition~U} \arrow[Rightarrow,r] & \begin{cases} \text{Condition~W}\\ \min\{k_2,k_3\} \geq n-d_1+2\\ \dim\spn \{ x_{a,\hat{1}} : a \in [n]\}=n\end{cases} \\
%\text{Condition~C} \arrow[Rightarrow,ur] & & &
%\end{tikzcd}
%\end{equation}
%
%\begin{equation}
%\begin{tikzcd}[row sep=tiny]
%\text{Condition~H} \arrow[Rightarrow,dr]  & &\\
% &  \hspace{1pt} \text{Condition~U} \arrow[Rightarrow,r] & \begin{cases} \text{Condition~W}\\ \min\{k_2,k_3\} \geq n-d_1+2\\ \dim\spn \{ x_{a,\hat{1}} : a \in [n]\}=n\end{cases} \\
%\text{Condition~C} \arrow[Rightarrow,ur] & &
%\end{tikzcd}
%\end{equation}

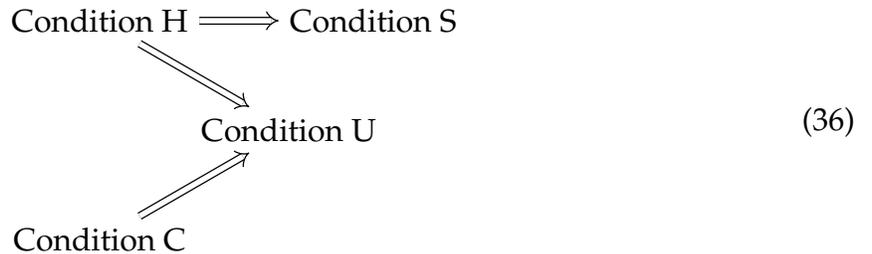
\begin{equation}\label{implications}
\begin{tikzpicture}[commutative diagrams/every diagram]
\node(dummy){};
\node[above=of dummy](H){\text{Condition H}};
\node[below=of dummy](C){\text{Condition C}};
\node[right=30pt of dummy](U){\text{Condition U}};
\node[right=30pt of H](S){\text{Condition S}};
%\node[right=30pt of U](W){$\begin{cases} \text{Condition~W}\\ \min\{k_2,k_3\} \geq n-d_1+2\\ \dim\spn \{ x_{a,\hat{1}} : a \in [n]\}=n.\end{cases}$};
\path[commutative diagrams/.cd, every arrow, every label]
(H) edge[commutative diagrams/Rightarrow]  (U)
(C) edge[commutative diagrams/Rightarrow] (U)
%(U) edge[commutative diagrams/Rightarrow] (W)
(H) edge[commutative diagrams/Rightarrow] (S);
\end{tikzpicture}
\end{equation}
All of the implications in~\eqref{implications} except (Condition H $\Rightarrow$ Condition S) were proven in~\cite{domanov2013uniqueness}. To see that Condition~H $\Rightarrow$ Condition S, note that for any subset $S\subseteq [n]$ with $2 \leq \abs{S} \leq n$, the condition $k_1 \geq 2$ implies
\begin{align}
d_{1}^S \geq \max\{2,d_{1}-(n-\abs{S})\},
\end{align}
so by Condition~H,
\begin{align}
d_{1}^S+d_{2}^S+d_{3}^S &\geq \max\{2,d_{1}-(n-\abs{S})\} + \abs{S}+\min\{\abs{S},n-d_{1}+2\}\\
&\geq 2 \abs{S}+2,
\end{align}
and Condition S holds. It is easy to find examples that certify {Condition~C $\not\Rightarrow$ Condition~S}.  By Example~\ref{uniqueness_example}, {Condition~S $\not\Rightarrow$ Condition~U}. In~\cite{domanov2013uniqueness} it is asked whether Condition~H $\Rightarrow$ Condition~C. Condition~U is theoretically computable, as it can be phrased as an ideal membership problem, however we are unaware of an efficient implementation. By comparison, Conditions~C,~H, and~S are easy to check.

In the case of three subsystems, our Theorem~\ref{k-gen} states that Condition~S implies uniqueness. Since {Condition~H $\Rightarrow$ Condition~S}, then a corollary to Theorem~\ref{k-gen} is that Condition~H implies uniqueness. Similarly, Theorem~\ref{uniqueness} below states that Condition~U + extra assumptions implies uniqueness. By~\eqref{implications}, this implies that Condition~H + the same extra assumptions implies uniqueness, and similarly, Condition~C + the same extra assumptions implies uniqueness. Since we have proven that Condition~H alone implies uniqueness, it is natural to ask whether Conditions~C or~U alone imply uniqueness. We reiterate this line of reasoning in Section~\ref{new_conjecture}, and pose this question formally.

\subsubsection{Synthesizing the uniqueness criteria of DLS}

The following theorem contains every uniqueness criterion of DLS for which we are aware of an efficient implementation. This theorem is stated in terms of Condition~U to maintain generality, however only the implied statements in which Condition~U is replaced by Conditions~H or~C have an efficient implementation. Note that our Theorem~\ref{k-gen} generalizes the Condition~H version of this theorem, to the statement that Condition~S alone implies uniqueness (so in particular, Condition~H alone implies uniqueness).
%Our Conjecture~\ref{k-gen} would generalize the version of Theorem~\ref{uniqueness} with Condition~U replaced by Condition~H, to the statement that Condition~H alone certifies uniqueness.

%We have omitted a few statements involving Condition~W that appear in~\cite{domanov2013uniqueness,domanov2013uniqueness2}, as they are more complicated to state and even more difficult to check than the statements involving Condition~U.

{\begin{theorem}\label{uniqueness}
Suppose that Condition~U holds, and any one of the following conditions holds: 
\begin{enumerate}[align=left]
%\item \hfill $k_{\sigma(1)}+\max\{ \min\{k_{\sigma(2)},k_{\sigma(3)}-1\},\min\{k_{\sigma(2)}-1,k_{\sigma(3)}\}\}\geq n+1$ \hfill 
%\refstepcounter{equation}(\textup{\theequation})
\item \leavevmode\vspace{-\dimexpr\baselineskip + \topsep + \parsep} \begin{center}$k_{1}+\min\{k_{2},k_{3}-1\}\geq n+1.$ \end{center}
\item It holds that $k_2 \geq 2$ and for all $\alpha \in \field^n$,
\begin{align}
\setft{rank}\Big[\sum_{a \in [n]} \alpha_a x_{a,1}\otimes x_{a,3}\Big]\geq \min\{\omega(\alpha),n-d_2+2\}.
\end{align}
(Note that this is just Condition U with the first subsystem replaced by the second).
\item There exists a subset $S\subseteq [n]$ with $0 \leq |S| \leq d_{1}$ such that the following three conditions hold:
\begin{enumerate}
\item \leavevmode\vspace{-\dimexpr\baselineskip + \topsep + \parsep} \begin{center} $d_{1}^S=|S|.$ \end{center}
\item \leavevmode\vspace{-\dimexpr\baselineskip + \topsep + \parsep} \begin{center} $d_{2}^{[n]\setminus S}=n-|S|.$\end{center}
\item For any linear map $\Pi \in \Lin(\V_{1})$ with ${\ker(\Pi) = \spn \{x_{a ,1}: a \in S\}}$, scalars $\alpha_1,\dots, \alpha_n \in \field$, and index $b \in [n]\setminus S$ such that
\begin{align}
\sum_{a \in [n]\setminus S} \alpha_a & \Pi x_{a, 1} \otimes x_{a, 3} = \Pi x_{b, 1} \otimes z
\end{align}
for some $z \in \V_{\sigma(3)}$, it holds that $\omega(\alpha)\leq 1$.
\end{enumerate}

\item There exists a permutation $\tau \in S_n$ for which the matrix
\begin{align}
X_{1}^\tau = (x_{\tau(1),1},\dots, x_{\tau(n),n})
\end{align}
has reduced row echelon form
\begin{align}\label{Y}
Y=\left[\begin{array}{@{}c | c@{}}
  \begin{matrix}
  1 & & \\
   & \ddots &\\
   && 1
  \end{matrix}
  & \begin{matrix}
   & & \\
   & {\normalfont \Huge Z} &\\
   && 
  \end{matrix}
%  \\
%  \hline
%   &
\end{array}\right],
\end{align}
where $Z \in \Lin(\field^{n-d_{1}},\field^{d_{1}})$ and the blank entries are zero. Furthermore, for each $a \in [d_{1}-1]$, the columns of the submatrix of $Y$ with row index $\{a,a+1,\dots, d_{1}\}$ and column index $\{a, a+1,\dots, n\}$ have k-rank at least two.
\item \leavevmode\vspace{-\dimexpr\baselineskip + \topsep + \parsep} \begin{center} $k_{1}=d_{1}.$ \end{center}
\item For all $\alpha \in \field^n$,
\begin{align}
\setft{rank}\Big[\sum_{a \in [n]} \alpha_a x_{a,2}\otimes x_{a,3}\Big]\geq \min\{\omega(\alpha),n-k_1+2\}.
\end{align}
(Note that this is a stronger statement than Condition U, as it replaces the quantity ${n-d_1+2}$ with the possibly larger quantity $n-k_1+2$.)
\end{enumerate}
Then $\sum_{a\in[n]} x_a$ constitutes a unique tensor rank decomposition.
\end{theorem}}

For each $i \in [5]$, we will refer to Theorem~\hyperref[uniqueness]{\ref{uniqueness}.i} as the statement that Condition~U and the $i$-th condition appearing in Theorem~\ref{uniqueness} imply uniqueness. Theorems~\hyperref[uniqueness]{\ref{uniqueness}.1} and~\hyperref[uniqueness]{\ref{uniqueness}.2} are Corollary~1.23 and Proposition~1.26 in \cite{domanov2013uniqueness2,domanov2014canonical}. The Condition~C version of Theorem~\hyperref[uniqueness]{\ref{uniqueness}.3} is stated in Theorem~2.2 in \cite{sorensen2015coupled}, although the proof is contained in \cite{domanov2013uniqueness,domanov2013uniqueness2,sorensen2015new}. Condition 3b in Theorem~\ref{uniqueness} can be formulated as checking the rank of a certain matrix (see \cite{sorensen2015coupled}). Theorem~\hyperref[uniqueness]{\ref{uniqueness}.4} is a new result that we will prove (see Proposition~\ref{invariant_condition_4} for a coordinate-free statement). The Condition~C version of Theorems~\hyperref[uniqueness]{\ref{uniqueness}.5} and~\hyperref[uniqueness]{\ref{uniqueness}.6} are Theorems~1.6 and~1.7 in~\cite{domanov2014canonical}. It is easy to see that our Theorem~\hyperref[uniqueness]{\ref{uniqueness}.4} contains Theorem~\hyperref[uniqueness]{\ref{uniqueness}.5}, which in turn contains Theorem~\hyperref[uniqueness]{\ref{uniqueness}.6}, by the arguments used in~\cite{domanov2014canonical}.

Most of these statements have previously only been formulated for $\field=\real$ or $\field=\complex$, however in all of these cases the proof can be adapted to hold over an arbitrary field. The first step in proving all of these statements is to show that Condition~U implies uniqueness in the first subsystem. This is Proposition~4.3 in~\cite{domanov2013uniqueness}, and it is proven using Kruskal's permutation lemma \cite{kruskal1977three} (the proof of the permutation lemma in~\cite{landsberg2012tensors} holds word-for-word over an arbitrary field). In fact, uniqueness in the first subsystem holds even with the assumption $k_1 \geq 2$ removed from Condition~U~\cite{domanov2013uniqueness}.

A less-restrictive condition than Condition~U, which we would call Condition~W, also appears in \cite{domanov2013uniqueness,domanov2013uniqueness2}, and is the same as Condition~U except that it only requires~\eqref{Umeq} to hold when ${\alpha = (f(x_{1,1}),\dots, f(x_{n,1}))}$ for some linear functional $f \in \V_1^*$. We note that Theorem~\ref{uniqueness} also holds with Condition~U replaced by Condition~W. Although the Condition~W version of Theorem~\ref{uniqueness} is slightly stronger than the Condition~U version, we are not aware of an efficient algorithm to check either Condition~U or Condition~W, and the existence of such an algorithm seems unlikely.

We conclude this subsection by proving Theorem~\hyperref[uniqueness]{\ref{uniqueness}.4}. For this we require the following proposition, which restates Condition~4 in a coordinate-free manner.

\begin{prop}\label{invariant_condition_4}
Condition 4 in Theorem~\ref{uniqueness} holds if and only if there exists a permutation $\tau \in S_n$ such that for each $a \in [d_1-1]$ there is a linear operator $\Pi_a \in \Lin(\V_1)$ for which
\begin{align}
\Pi_a (x_{\tau(b),1})=0
\end{align}
for all $b \in [a-1]$, and
\begin{align}\label{k-rank-thingy}
\krank( \Pi_a x_{\tau(a),1},\dots,\Pi_a x_{\tau(n),1})\geq 2.
\end{align}
\end{prop}
\begin{proof}
Assume without loss of generality that $\V_1=\field^{d_1}$. To see that the first statement implies the second, for each $a\in [d_1-1]$ let $\Pi_a=D_a P$, where $P\in \Lin(\field^{d_1})$ is the invertible matrix for which $PX_1^\tau=Y$, and $D_a \in \Lin(\field^{d_1})$ is the diagonal matrix with the first $a-1$ entries zero and the remaining entries $1$. It is easy to verify that~\eqref{k-rank-thingy} holds.

%\begin{align}\label{D}
%D_a=\left[\begin{array}{@{}c | c@{}}
%  \begin{matrix}
%  0 & & &&&\\
%   & \ddots &&&&\\
%   && 0&&&\\
%   &&&1&&\\
%   &&&&\ddots&\\
%   &&&&&1
%  \end{matrix}
%  & \begin{matrix}
% &&\\
% &&\\
% &&\\
% &&\\
% &&\\
% &&
%  \end{matrix}
%\end{array}\right],
%\end{align}
%where the lefthand block is a $d_1\times d_1$ diagonal matrix with the first $a-1$ entries zero and the remaining entries $1$, and the last $n-d_1$ columns zero.

Conversely, suppose that the reduced row echelon form of $X_1^\tau$, given by $P X_1^\tau$ for some invertible matrix $P \in \Lin(\field^{d_1})$, does not have the specified form. Then there exists ${a\in [d_1-1]}$ for which the columns of $D_a P X_1^\tau$ have k-rank at most one. Any matrix $\Pi_a\in \Lin(\field^{d_1})$ for which $\Pi_a (x_{\tau(b),1})=0$ for all $b \in [a-1]$ satisfies
\begin{align}
\Pi_a =\Pi_a P^{-1} D_a P.
\end{align}
Since the k-rank is non-increasing under matrix multiplication from the left,~\eqref{k-rank-thingy} does not hold.
\end{proof}

With Proposition~\ref{invariant_condition_4} in hand, we can now prove Theorem~\ref{uniqueness}.4.

\begin{proof}
[Proof of Theorem~\ref{uniqueness}.4]
The question of whether or not the decomposition $\sum_{a\in [n]} x_a$ constitutes a unique tensor rank decomposition is invariant under permutations $\tau \in S_n$ of the tensors, so it suffices to prove the statement under the assumption that the permutation $\tau$ appearing in Condition 4 is trivial. We prove the statement by induction on $d_1$. If $d_1=2$, then Condition~U implies $k_2=k_3=n$, so uniqueness follows from Kruskal's theorem. For $d_1>2$, suppose $\sum_{a\in [n]} x_a =\sum_{a\in [r]} y_a$ for some non-negative integer $r \leq n$ and multiset of product tensors
\begin{align}
\{y_a : a \in [r]\} \subseteq \pro{\V_1: \V_2 : \V_3}.
\end{align}
By Proposition~4.3 in~\cite{domanov2013uniqueness} (or rather, the extension of this result to an arbitrary field), $r=n$, and there exists a permutation $\sigma \in S_n$ and nonegative integers $\alpha_1,\dots, \alpha_n \in \field^\times$ such that $\alpha_a x_{a,1}=y_{\sigma(a),1}$ for all $a \in [n]$. Let $\Pi_1 \in \Lin(\V_1)$ be any operator for which ${\ker(\Pi_1)=\spn\{x_{a,1}\}}$ and~\eqref{k-rank-thingy} holds (recall that $\tau$ is trivial). Then
\begin{align}
\sum_{a \in [n] \setminus \{1\}} (\Pi_1 x_{a,1})\otimes x_{a,2} \otimes x_{a,3}=\sum_{a \in [n] \setminus \{1\}} (\alpha_{a} \Pi_1  x_{a,1})\otimes y_{\sigma(a),2} \otimes y_{\sigma(a),3}.
\end{align}
Now, $\dim\spn\{\Pi_1 x_{a,1}: a \in [n]\setminus\{1\}\}=d_1-1$, and Condition~U again holds for the multiset of product tensors 
\begin{align}
\{(\Pi_1 x_{a,1}) \otimes x_{a,2} \otimes x_{a,3} : a \in [n] \setminus \{1\}\}.
\end{align}
Furthermore, these product tensors again satisfy Condition 4 of Theorem~\ref{uniqueness}, so by the induction hypothesis
\begin{align}
(\Pi_1 x_{a,1}) \otimes x_{a,2} \otimes x_{a,3}=(\alpha_a \Pi_1 x_{a,1}) \otimes y_{\sigma(a),2} \otimes y_{\sigma(a),3}\quad \text{for all}\quad a \in [n]\setminus \{1\}.
\end{align}
It follows that $x_a=y_{\sigma(a)}$ for all $a \in [n] \setminus\{1\}$, so $x_1=y_{\sigma(1)}$. This completes the proof.
\end{proof}

\subsection{Conjectural generalization of all uniqueness criteria of DLS}\label{new_conjecture}
In the case of three subsystems, our generalization of Kruskal's theorem states that Condition~S implies uniqueness. Since {Condition~H $\Rightarrow$ Condition~S}, then a corollary to Theorem~\ref{k-gen} is that Condition~H implies uniqueness. Similarly, Theorem~\ref{uniqueness} above states that Condition~U + extra assumptions implies uniqueness, which implies that Condition~H + the same extra assumptions implies uniqueness. Since we have proven that Condition~H alone implies uniqueness, it is natural to ask whether Condition~U alone implies uniqueness. We now state this question formally. A positive answer to Question~\ref{conjecture_U} would generalize and unify all of the uniqueness criteria of DLS (synthesized in Theorem~\ref{uniqueness}) into a single, elegant statement.

\begin{question}\label{conjecture_U}
Does Condition~U imply that $\sum_{a \in [n]} x_a$ constitutes a unique tensor rank decomposition?
\end{question}

%As further evidence for Conjecture
%
%Recall that, by Theorem~\ref{k-gen}, Condition~H alone certifies uniqueness. This generalizes the version of Theorem~\ref{uniqueness} with Condition~U replaced by Condition~H. A natural question that then arises is whether Condition~U alone certifies uniqueness. Theorems~\hyperref[uniqueness]{\ref{uniqueness}.4} and~\hyperref[uniqueness]{\ref{uniqueness}.5} are distinguished among the results in Theorem~\ref{uniqueness}, in that the extra conditions imposed beyond Condition~U concern only the first subsystem. One can view these results as further evidence that Condition~U alone certifies uniqueness, as they show that no further conditions on the second and third subsystems are necessary for uniqueness.

\section{Appendix}
In this appendix we prove Theorem~\ref{k_arbitrary_old}. The proof is very similar to that of Theorem~\ref{k_arbitrary}.
\begin{proof}[Proof of Theorem~\ref{k_arbitrary_old}]
For each $a \in [r]$, let $x_{n+a}=-y_a$, and let $T_1 \sqcup \dots \sqcup T_t=[n+r]$ be the index sets of the decomposition of $\{x_a : a \in [n+r]\}$ into connected components. Note that for each $p \in [t]$, if
\begin{align}
\bigabs{T_p \cap [n+r]\setminus [n]}\leq \bigabs{T_p \cap [n]},
\end{align}
then $\bigabs{T_p \cap [n]} \leq s$, otherwise $\{x_a : a \in T_p\}$ would split. Assume without loss of generality that
\begin{align}
\bigabs{T_1 \cap [n]}-\bigabs{T_1 \cap [n+r]\setminus [n]} &\geq \bigabs{T_2 \cap [n]}-\bigabs{T_2 \cap [n+r]\setminus [n]}\\
 &\;\; \vdots \\
 & \geq \bigabs{T_{{t}} \cap [n]}-\bigabs{T_{{t}} \cap [n+r]\setminus [n]},
\end{align}
If
\begin{align}
\bigabs{T_{1}\cap [n]} \geq \bigabs{T_{1} \cap [n+r]\setminus [n]},
\end{align}
then let $\tilde{l}\in [t]$ be the largest integer for which
\begin{align}\label{inequality_old}
\bigabs{T_{\tilde{l}}\cap [n]} \geq \bigabs{T_{\tilde{l}} \cap [n+r]\setminus [n]}.
\end{align}
Otherwise, let $\tilde{l}=0$. Then for all $p \in [t]\setminus [\tilde{l}]$ it holds that
\begin{align}\label{strict_inequality_old}
\bigabs{T_p \cap [n]} < \bigabs{T_p \cap [n+r]\setminus [n]}.
\end{align}
To complete the proof, we will show that $\tilde{l} \geq l$, for then we can take $Q_p=T_p \cap [n]$ and $R_p=T_p \cap [n+r]\setminus [n]$ for all $p \in [l]$ to conclude.

Suppose toward contradiction that $\tilde{l}< l$. We will require the following two claims:

\begin{claim}\label{pigeonhole_claim_old}
It holds that $\tilde{l}<t$, $\bigceil{\frac{n- s \tilde{l} }{t-\tilde{l}}} \geq s+1$, and there exists $p \in [t]\setminus [\tilde{l}]$ for which
\begin{align}\label{pigeonhole}
\bigabs{T_p \cap [n]} \geq \biggceil{\frac{n- s \tilde{l} }{t-\tilde{l}}}.
\end{align} 
\end{claim}

\begin{claim}\label{inequality2_claim_old}
For all $p \in [t] \setminus [\tilde{l}]$, it holds that
\begin{align}\label{inequality2_old}
\bigabs{T_p \cap [n+r]\setminus [n]} \leq \bigabs{T_p \cap [n]} +(r-n)+(s+1)\tilde{l} -t+1.
\end{align}
\end{claim}

Before proving these claims, we first use them to complete the proof of the theorem. Let $p \in [t]\setminus [\tilde{l}]$ be as in Claim~\ref{pigeonhole_claim_old}. Then,
\begin{align}
\abs{T_p} &= \bigabs{T_p \cap [n]}+\bigabs{T_p \cap [n+r]\setminus [n]}\\
		&\leq 2 \bigabs{T_p \cap [n]}+ r-n+(s+1)\tilde{l} -t+1\\
		&\leq 2 \bigabs{T_p \cap [n]}+ r-n + s \tilde{l} - \biggceil{\frac{n-s \tilde{l}}{\bigabs{T_p \cap [n]}}}+1\\
		%&\leq 2 \bigabs{T_p \cap [n]}+ r-n+s(l-1)- \frac{n-s (l-1)}{\bigabs{T_p \cap [n]}}+1\\
		&\leq 2 \bigabs{T_p \cap [n]}+ (r-n+q-s)- \biggceil{\frac{n-q+s}{\bigabs{T_p \cap [n]}}}+1\\
		& \leq \sum_{j=1}^m (d_j^{T_p \cap [n]}-1)+1,
\end{align}
where the first line is obvious, the second follows from Claim~\ref{inequality2_claim_old}, the third follows from Claim~\ref{pigeonhole_claim_old}, the fourth follows from $\tilde{l}< l$, and the fifth follows from the assumptions of the theorem and the fact that $\abs{T_p \cap [n] } \geq s+1$. So $\{x_a : a \in T_p\}$ splits, a contradiction. This completes the proof, modulo proving the claims.

\begin{proof}[Proof of Claim~\ref{pigeonhole_claim}]
\renewcommand\qedsymbol{$\triangle$}
To prove the claim, we first observe that $n>st$. Indeed, if $n \leq st$, then
\begin{align}
r& \geq \sum_{p=\tilde{l}+1}^t \bigabs{T_p \cap [n+r]\setminus [n]}\\
&\geq\sum_{p=\tilde{l}+1}^t \left(\bigabs{T_p \cap [n]}+1\right)\\
&=n-\bigabs{(T_1 \sqcup \dots \sqcup T_{\tilde{l}})\cap [n]}+t-\tilde{l}\\
&\geq n+t - (s+1)\tilde{l}\\
& \geq n+\biggceil{\frac{n}{s}-(s+1)(q/s-1)}\\
&=  \biggceil{\left(\frac{s+1}{s}\right)(n-q+s)},\\
\end{align}
%\begin{align}
%r &< \left(\frac{s+1}{s}\right)(n-q+s)\\
%  &=n+\frac{n}{s}-(s+1)(q/s-1)\\
%  &\leq n+t - (s+1)\tilde{l}\\
%  &\leq n+t-\bigabs{(T_1 \sqcup \dots \sqcup T_{\tilde{l}})\cap [n]}-\tilde{l}\\
%  &=\sum_{p=\tilde{l}+1}^t (\abs{T_p \cap [n]}+1)\\
%  &\leq \sum_{p=\tilde{l}+1}^t \bigabs{T_p \cap [n+r]\setminus [n]}\\
%  &\leq r,
%  \end{align}
where the first line is obvious, the second follows from~\eqref{strict_inequality}, the third is obvious, the fourth follows from $\abs{T_p \cap [n]} \leq s$ for all $p \in [\tilde{l}]$, the fifth follows from $n\leq st$ and $\tilde{l}<l$, and the sixth is algebra. This contradicts the assumptions of the theorem, so it must hold that $n>st$.

Note that $\tilde{l}<t$, for otherwise we would have $n \leq st$ by the fact that $\bigabs{T_p \cap [n]} \leq s$ for all $p \in [\tilde{l}]$. To verify that $\bigceil{\frac{n- s \tilde{l} }{t-\tilde{l}}} \geq s+1$, it suffices to prove $\frac{n- s \tilde{l} }{t-\tilde{l}} > s,$ which follows from $n>st$. To verify~\eqref{pigeonhole}, since $\abs{T_p \cap [n]} \leq s$ for all $p \in [\tilde{l}]$, by the pigeonhole principle there exists $p \in [t]\setminus [\tilde{l}]$ for which
\begin{align}
\bigabs{T_p \cap [n]} \geq \biggceil{\frac{n- s \tilde{l} }{t-\tilde{l}}}.
\end{align}
This proves the claim.
\end{proof}

\begin{proof}[Proof of Claim~\ref{inequality2_claim_old}]
\renewcommand\qedsymbol{$\triangle$}
Suppose toward contradiction that the inequality~\eqref{inequality2_old} does not hold for some $\tilde{p}\in [t] \setminus [\tilde{l}]$. Then
\begin{align}
r &\geq \sum_{p=\tilde{l}+1}^t \bigabs{T_p \cap [n+r]\setminus [n]}\\
& \geq \sum_{p \neq \tilde{p}} \big(\bigabs{T_p \cap [n]}+1\big)+ \bigabs{T_{\tilde{p}} \cap [n]}+ (r-n)+(s+1)\tilde{l} -t+2\\
&= \sum_{p=\tilde{l}+1}^t \bigabs{T_p \cap [n]}+ (r-n)+s\tilde{l}+1\\
%& \geq n-s\tilde{l}+(r-n)+s\tilde{l}+1\\
&\geq r+1,
\end{align}
where the first three lines are obvious, and the fourth follows from~\eqref{inequality_old}, a contradiction.
\end{proof}
The proofs of Claims~\ref{pigeonhole_claim_old} and~\ref{inequality2_claim_old} complete the proof of the theorem.
\end{proof}

%-----------------------------------------------------------------------------%
\bibliographystyle{alpha}
\bibliography{tensor_rank}
%-----------------------------------------------------------------------------%

\end{document}